\documentclass[11pt,reqno]{amsart}
\usepackage[utf8]{inputenc} 
\usepackage[T1]{fontenc}    
\usepackage{hyperref}
\usepackage{amsfonts}       
\usepackage{microtype}      

\usepackage{amscd,amssymb,amsmath,amsthm}
\usepackage[arrow,matrix]{xy}
\usepackage{graphicx}
\usepackage{multicol}
\usepackage{subfigure}
\usepackage{epstopdf}
\usepackage{color}
\usepackage{soul}
\newtheorem{thm}{Theorem}
\newtheorem{defn}{Definition}
\newtheorem{lemma}{Lemma}
\newtheorem{pro}{Proposition}
\newtheorem{rk}{Remark}

\numberwithin{equation}{section} 

\begin{document}

\vspace{0.5in}
\renewcommand{\bf}{\bfseries}
\renewcommand{\sc}{\scshape}
\vspace{0.5in}

\title[Stability, Bifurcation, and Chaos Control]{Stability, Bifurcation, and Chaos Control in a Discrete-Time Phytoplankton-Zooplankton Model with Holling Type II and Type III Functional Responses}

\author{Sobirjon Shoyimardonov}

\address{S.K. Shoyimardonov$^{a,b}$ \begin{itemize}
\item[$^a$] V.I.Romanovskiy Institute of Mathematics, 9, University str.,
Tashkent, 100174, Uzbekistan;
\item[$^b$] National University of Uzbekistan,  4, University str., 100174, Tashkent, Uzbekistan.
\end{itemize}}
\email{shoyimardonov@inbox.ru}







\begin{abstract} In this paper, we investigate the dynamics of a discrete-time phytoplankton-zooplankton model where the predator functional response and toxin distribution functions follow both Holling Type II and Holling Type III forms simultaneously. We analyze the types of fixed points and the global stability of the system. Additionally, we prove the occurrence of a Neimark-Sacker bifurcation at the positive fixed point. The theoretical findings are validated through numerical simulations.
\end{abstract}

\maketitle

\section{Introduction}
In the ecosystem, studying the relationship between prey and predator and investigating their populations remains one of the pressing issues to this day. The earliest predator-prey models were developed in the early 20th century by Alfred Lotka and Vito Volterra. These models indicate that the relationship between predator and prey exhibits complex behavior. Among the most widely studied classes of predator-prey models are those featuring Holling type II and type III functional responses, which model predator consumption rates and their effect on prey populations. In particular, ocean ecosystems can also be considered as predator-prey models.

The study of ocean ecosystems, especially those involving plankton populations, is critical for understanding global ecological processes. In recent years, numerous studies have investigated the dynamics of these systems~\cite{Chen, Hen, Hong, Mac, Qiu, RSH, RSHV, Sajan, Shang, Tian}. Phytoplankton, as primary producers, contribute to carbon fixation through photosynthesis, while zooplankton, as primary consumers, play a vital role in marine food webs. The dynamics between these two groups are closely intertwined, with changes in one often leading to corresponding fluctuations in the other. Predator-prey models provide a useful framework to study their interactions, accounting for growth, predation, mortality, and environmental factors such as nutrient availability. To analyze such interactions, bifurcation theory serves as an essential mathematical tool.

Bifurcation theory is crucial for understanding how the qualitative behavior of a dynamical system changes as parameters are varied. Its applications span many scientific disciplines.  In physics, it is essential for analyzing critical phenomena like phase transitions, fluid instabilities, and the onset of turbulence in nonlinear systems. In chemistry, bifurcation theory is used to study complex reactions, such as the Belousov-Zhabotinsky reaction, offering insights into oscillatory and chaotic behavior driven by autocatalytic feedback. In economics and finance, bifurcations help model the emergence of complex phenomena like market crashes, business cycles, and macroeconomic tipping points.  In ecology, bifurcation analysis reveals how small changes in parameters can lead to significant shifts in predator-prey dynamics. Bifurcations occur in both continuous-time systems, typically described by differential equations, and discrete-time systems, which are represented by maps.

Continuous-time models have long served as a foundation in mathematical biology, offering insights into the instantaneous and continuous evolution of populations. However, in many ecological scenarios, processes and data collection occur at discrete intervals, such as seasonal breeding or periodic sampling. In these cases, discrete-time models provide a more natural and practical approach. They capture changes from one time step to the next and are especially amenable to computational analysis.

Discrete-time models often reveal dynamic behaviors that are less evident in continuous models. A notable example is the Neimark-Sacker bifurcation, the discrete analogue of the Hopf bifurcation. It occurs when a pair of complex-conjugate eigenvalues of the Jacobian crosses the unit circle, destabilizing a fixed point and leading to a closed invariant curve. The resulting quasi-periodic oscillations cause the system to orbit around the equilibrium, generating persistent fluctuations rather than convergence.

 Numerous studies have demonstrated bifurcations in predator-prey models~\cite{Cah, Chen3, Cheng, Ko, Li, Ma, Peng, Wang, Zhou}. For example, Cahit and Yasin~\cite{Cah} identified a Neimark-Sacker bifurcation in a discrete-time model with prey immigration. Li et al.~\cite{Li} observed Neimark-Sacker, period-doubling, and strong resonance bifurcations in a two-dimensional model. Chen et al.~\cite{Chen3} studied fold and flip bifurcations under the influence of the Allee effect.

Such bifurcations have significant ecological implications. In phytoplankton-zooplankton systems, Neimark-Sacker bifurcations may explain bloom-bust cycles, where rapid phytoplankton growth is followed by sharp declines due to grazing or toxin-mediated suppression of zooplankton. Many phytoplankton species produce toxic or allelopathic compounds that reduce zooplankton feeding efficiency and survival. Incorporating toxin production into predator-prey models adds nonlinear inhibitory effects, increasing the potential for oscillatory and chaotic behavior.

For instance, Edwards and Brindley~\cite{Ed} examined an NPZ (nutrient-phytoplankton-zooplankton) model and observed oscillatory behavior. In~\cite{Chatt2}, Chattopadhyay et al.\ proposed a mechanism to explain cyclic bloom dynamics using toxin liberation modeled by a fractional (Holling type II) function, while assuming a linear predator response. They analyzed the occurrence of Hopf bifurcation and suggested that toxins may play a key role in bloom termination. Later, in~\cite{Chatt}, the model was generalized as follows:

\begin{equation}\label{chat}
\left\{\begin{aligned}
&\frac{dP}{dt}=bP\left(1-\frac{P}{K}\right)-\alpha f(P)Z,\\
&\frac{dZ}{dt}=\beta f(P)Z - rZ - \theta g(P)Z,
\end{aligned}\right.
\end{equation}
where $P$ and $Z$ denote the densities of phytoplankton and zooplankton, respectively. The parameters $\alpha, \beta > 0$ represent predation and conversion rates; $b, K > 0$ are the intrinsic growth rate and carrying capacity of phytoplankton; $r > 0$ is the zooplankton mortality rate. The function $f(P)$ is the predator's functional response, and $g(P)$ describes toxin distribution. The parameter $\theta > 0$ denotes the rate of toxin liberation by phytoplankton.

In~\cite{Chatt}, nine biologically meaningful combinations of $f(P)$ and $g(P)$ were studied. The authors showed that toxin release can suppress blooms. A widely used functional form is
\[
\frac{P^h}{1 + cP^h},
\]
which corresponds to Holling type II when $h = 1$ and Holling type III when $h = 2$, unifying multiple response types within a single framework.

In~\cite{Chen}, model~\eqref{chat} was analyzed with $f(P)$ as either Holling type II or III, and a linear toxin function $g(P) = P$. In~\cite{SH}, a Neimark-Sacker bifurcation was established for the discrete-time version of this model with $f(P)$ of Holling type II and linear toxin liberation.

In this work, we consider $f(P) = g(P) = \frac{P^h}{1 + cP^h}$ (for $h = 1, 2$), capturing both Holling type II and III responses. While real systems often exhibit one dominant response type, including both allows comparative analysis of their ecological implications within a unified framework.

We non-dimensionalize model~\eqref{chat} by introducing
\[
\overline{t}=bt, \quad \overline{u}=\frac{P}{K}, \quad \overline{v}=\frac{\alpha K^{h-1}Z}{b}, \quad \overline{c}=cK^h, \quad \overline{\beta}=\frac{\beta K^h}{b}, \quad \overline{r}=\frac{r}{b}, \quad \overline{\theta}=\frac{\theta K^h}{b},
\]
and dropping the overlines yields the simplified system:
\begin{equation}\label{chenn}
\left\{\begin{aligned}
&\frac{du}{dt} = u(1 - u) - \frac{u^h v}{1 + cu^h}, \\
&\frac{dv}{dt} = \frac{\beta u^h v}{1 + cu^h} - rv - \frac{\theta u^h v}{1 + cu^h}.
\end{aligned}\right.
\end{equation}

Applying the Euler discretization method with time step $\Delta t = 1$ (e.g., one day), appropriate for plankton ecology, we obtain the discrete-time system:
\begin{equation}\label{h12}
\begin{cases}
u^{(1)} = u(2 - u) - \dfrac{u^h v}{1 + cu^h}, \\[2mm]
v^{(1)} = \dfrac{\gamma u^h v}{1 + cu^h} + (1 - r)v,
\end{cases}
\end{equation}
where $\gamma = \beta - \theta$. We assume $r, c > 0$ and $\gamma \in \mathbb{R}$.

By varying $\gamma$, we identify a Neimark-Sacker bifurcation at a positive fixed point. This bifurcation marks a transition from stability to oscillatory dynamics, reflecting the ecological shift from steady-state to periodic interactions. It reveals how changes in zooplankton efficiency ($\beta$) or phytoplankton toxicity ($\theta$) can lead to dynamic shifts in population behavior. Thus, $\gamma$ serves as a critical bifurcation parameter.

In this study, we investigate the dynamics of model~\eqref{h12}. Section~2 analyzes the fixed points. Section~3 demonstrates the occurrence of a Neimark-Sacker bifurcation and discusses chaos control. Section~4 explores global stability, and Section~5 verifies the results through numerical simulations. Finally, Section~6 presents concluding remarks.

\section{stability of fixed points}

It is obvious that, system (\ref{h12}) always has two nonnegative equilibria $E_0=(0; 0)$ and $E_1=(1; 0)$. From $v^{(1)}=v$ we get that $\overline{E}=(\overline{u}; \overline{v})$ is a positive fixed point  of system (\ref{h12}) if and only if  $\gamma>r(1+c),$ where:

\begin{equation}\label{pos}
\overline{u}^{h}=\frac{r}{\gamma-rc}, \ \ \overline{v}=\frac{(1-\overline{u})(1+c\overline{u}^{h})}{\overline{u}^{h-1}}.
\end{equation}

\begin{defn}\label{def1}
Let $E(x,y)$ be a fixed point of the operator $F:\mathbb{R}^{2}\rightarrow\mathbb{R}^{2}$, and let $\lambda_1, \lambda_2$ be the eigenvalues of the Jacobian matrix $J = J_{F}$ at the point $E(x,y)$. The classification of the fixed point is as follows:

\begin{itemize}
    \item[(i)] If $|\lambda_1|<1$ and $|\lambda_2|<1$, then $E(x,y)$ is called an attractive.
    \item[(ii)] If $|\lambda_1|>1$ and $|\lambda_2|>1$, then $E(x,y)$ is called a repelling.
    \item[(iii)] If $|\lambda_1|<1$ and $|\lambda_2|>1$ (or $|\lambda_1|>1$ and $|\lambda_2|<1$), then $E(x,y)$ is called a saddle.
    \item[(iv)] If either $|\lambda_1|=1$ or $|\lambda_2|=1$, then $E(x,y)$ is called non-hyperbolic.
\end{itemize}
\end{defn}

Before analyzing the fixed points we give the following useful lemma.

\begin{lemma}[Lemma 2.1, \cite{Cheng}]\label{lem1}
Let \( F(\lambda) = \lambda^2 + B\lambda + C \), where \( B \) and \( C \) are two real constants. Suppose \( \lambda_1 \) and \( \lambda_2 \) are the roots of \( F(\lambda) = 0 \). If \( F(1) > 0 \), then the following statements hold:
     \begin{itemize}
        \item[(i.1)] \( |\lambda_1| < 1 \) and \( |\lambda_2| < 1 \) if and only if \( F(-1) > 0 \) and \( C < 1 \).
        \item[(i.2)] \( \lambda_1 = -1 \) and \( \lambda_2 \neq -1 \) if and only if \( F(-1) = 0 \) and \( B \neq 2 \).
        \item[(i.3)] \( |\lambda_1| < 1 \) and \( |\lambda_2| > 1 \) if and only if \( F(-1) < 0 \).
        \item[(i.4)] \( |\lambda_1| > 1 \) and \( |\lambda_2| > 1 \) if and only if \( F(-1) > 0 \) and \( C > 1 \).
        \item[(i.5)] \( \lambda_1 \) and \( \lambda_2 \) are a pair of conjugate complex roots with \( |\lambda_1| = |\lambda_2| = 1 \) if and only if \( -2 < B < 2 \) and \( C = 1 \).
        \item[(i.6)] \( \lambda_1 = \lambda_2 = -1 \) if and only if \( F(-1) = 0 \) and \( B = 2 \).
    \end{itemize}

 \end{lemma}

The following theorem classifies all fixed points of the operator (\ref{h12}).

\begin{thm} \label{thm1}
For the fixed points of \eqref{h12}, the following statements hold:

\begin{itemize}
    \item[(i)] The stability of $E_0$ is given by:
    \[
    E_0 = \begin{cases}
        \text{nonhyperbolic}, & \text{if } r = 2, \\
        \text{saddle}, & \text{if } 0 < r < 2, \\
        \text{repelling}, & \text{if } r > 2.
    \end{cases}
    \]

    \item[(ii)] The stability of $E_1$ is determined as follows:
    \[
    E_1 = \begin{cases}
        \text{nonhyperbolic}, & \text{if } \gamma = (r-2)(1+c) \text{ or } \gamma = r(1+c), \\
        \text{attractive}, & \text{if } (r-2)(1+c) < \gamma < r(1+c), \\
        \text{saddle}, & \text{otherwise}.
    \end{cases}
    \]

    \item[(iii)] The stability of $\overline{E}$ depends on the function $q(\overline{u})$:
    \[
    \overline{E} = \begin{cases}
        \text{nonhyperbolic}, & \text{if } q(\overline{u}) = 1, \\
        \text{attractive}, & \text{if } q(\overline{u}) < 1, \\
        \text{reppelling}, & \text{if } q(\overline{u}) > 1.
    \end{cases}
    \]
    where $\overline{u}$ is defined as in (\ref{pos}) and
    \[
    q(\overline{u}) = \frac{(1 - \overline{u})(2 - h + rh + 2c\overline{u}^h)}{1 + c\overline{u}^h}.
    \]
\end{itemize}

\end{thm}

\begin{proof}
The Jacobian matrix of the operator \eqref{h12} is given by
\begin{equation}\label{jacc}
J(u,v) = \begin{bmatrix}
2 - 2u - \frac{h v u^{h-1}}{(1 + c u^h)^2} & -\frac{u^h}{1 + c u^h} \\
\frac{\gamma h v u^{h-1}}{(1 + c u^h)^2} & \frac{\gamma u^h}{1 + c u^h} + 1 - r
\end{bmatrix}.
\end{equation}

(i) Evaluating the Jacobian at the origin, we obtain
\[
J(0,0) = \begin{bmatrix}
2 & 0 \\
0 & 1 - r
\end{bmatrix}.
\]
The eigenvalues of this matrix are $2$ and $1 - r$. From this, the stability of $E_0$ follows directly.

(ii) At the fixed point $E_1$, the Jacobian is given by
\[
J(1,0) = \begin{bmatrix}
0 & -\frac{1}{1 + c} \\
0 & \frac{\gamma}{1 + c} + 1 - r
\end{bmatrix}.
\]
The eigenvalues are $\lambda_1 = 0$ and $\lambda_2 = \frac{\gamma}{1 + c} + 1 - r$. By analyzing the conditions $|\lambda_2| = 1$, $|\lambda_2| < 1$, and $|\lambda_2| > 1$, we derive the stability criteria as stated in assertion (ii).

(iii) Now, we analyze the stability of the unique positive fixed point $\overline{E} = (\overline{u}, \overline{v})$, where $\overline{u}$ and $\overline{v}$ are defined as in \eqref{pos}. Simplifying the Jacobian \eqref{jacc}, we obtain

\begin{equation}\label{jacpos}
J(\overline{u}, \overline{v}) = \begin{bmatrix}
\frac{(1 - \overline{u})(2 - h + 2c \overline{u}^h)}{1 + c \overline{u}^h} & -\frac{r}{\gamma} \\
\frac{\gamma h (1 - \overline{u})}{1 + c \overline{u}^h}  & 1
\end{bmatrix}.
\end{equation}

The characteristic polynomial is given by
\begin{equation}\label{charpol}
F(\lambda, \overline{u}) = \lambda^2 - p(\overline{u}) \lambda + q(\overline{u}),
\end{equation}
where
\begin{equation}\label{pq}
p(\overline{u}) = 1 + \frac{(1 - \overline{u})(2 - h + 2c \overline{u}^h)}{1 + c \overline{u}^h}, \quad
q(\overline{u}) = \frac{(1 - \overline{u})(2 - h + rh + 2c \overline{u}^h)}{1 + c \overline{u}^h}.
\end{equation}
Evaluating at $\lambda = 1$ and $\lambda = -1$, we obtain
\[
F(1, \overline{u}) = \frac{rh(1 - \overline{u})}{1 + c \overline{u}^h} > 0
\]
and
\[
F(-1, \overline{u}) = 2 + \frac{(1 - \overline{u})(4 - 2h + rh + 4c \overline{u}^h)}{1 + c \overline{u}^h} > 0.
\]
Thus, by Lemma \ref{lem1}, the nature of the positive fixed point depends on the value of $q(\overline{u})$. This completes the proof.
\end{proof}

\section{bifurcation analysis and chaos control}

\subsection{Neimark-Sacker bifurcation}

In this subsection, we identify the conditions under which a Neimark-Sacker bifurcation occurs at the unique positive fixed point
$\overline{E} = (\overline{u}, \overline{v}).$
According to Theorem~\ref{thm1}, the Jacobian matrix at the fixed point \( \overline{E} \) has a pair of complex conjugate eigenvalues with modulus 1 if $q(\overline{u}) = 1,$
where \( q(\overline{u}) \) is defined in (\ref{pq}). We choose \( \gamma \) as the bifurcation parameter and denote by \( \gamma_0 \) a solution of the equation $q(\overline{u}) = 1.$

To address the ecological interpretation of the bifurcation parameter $\gamma$, we note that $\gamma = \beta - \theta$, where $\beta$ denotes the conversion efficiency of zooplankton and $\theta$ represents the rate of toxin release by phytoplankton. Biologically, $\gamma$ captures the net energetic gain of zooplankton from consuming phytoplankton, balancing the nutritional benefit against the harmful effects of toxins. An increase in $\gamma$ may reflect either an improvement in conversion efficiency or a reduction in toxicity, or both. To avoid ambiguity, we consider biologically relevant scenarios in which either $\beta$ or $\theta$ is held constant. For instance, fixing $\beta$ allows us to interpret variations in $\gamma$ as changes in toxin production, while fixing $\theta$ highlights the impact of improved assimilation efficiency. Additionally, we acknowledge that in natural systems, both $\beta$ and $\theta$ may vary simultaneously, for example, under environmental stress or evolutionary adaptation, leading to more complex shifts in $\gamma$. In such cases, $\gamma$ serves as a useful aggregate measure summarizing the combined ecological effects of these changes. This broader perspective reinforces the relevance of $\gamma$ as a meaningful bifurcation parameter in our analysis.

In the following, we prove that the fixed point \( \overline{E} \) undergoes a Neimark-Sacker bifurcation when the parameter \( \gamma \) varies within a small neighborhood of \( \gamma_0 \). To proceed, we consider the following steps.

\textbf{The first step}. We introduce the variable changes $x=u-\overline{u},$ $y=v-\overline{v}$, which shift
the fixed point \( \overline{E} = (\overline{u}, \overline{v}) \) to the origin, and transform system $\ref{h12}$ into

\begin{equation}\label{bif1}
\left\{\begin{aligned}
&x^{(1)}=(x+\overline{u}) (2-x-\overline{u})-\frac{(x+\overline{u})^h(y+\overline{v})}{1+c(x+\overline{u})^h}-\overline{u}\\
&y^{(1)}=(y+\overline{v})\left(\frac{\gamma(x+\overline{u})^h}{1+c(x+\overline{u})^h}+1-r\right)-\overline{v}.
\end{aligned}\right.
\end{equation}

\textbf{The second step}. We give a small perturbation $\gamma^*$ to the parameter $\gamma,$ such that $\gamma=\gamma_0+\gamma^*.$ Then the perturbed system (\ref{bif1}) can then be expressed as:

\begin{equation}\label{bif2}
\left\{\begin{aligned}
&x^{(1)}=(x+\overline{u}) (2-x-\overline{u})-\frac{(x+\overline{u})^h(y+\overline{v})}{1+c(x+\overline{u})^h}-\overline{u}\\
&y^{(1)}=(y+\overline{v})\left(\frac{(\gamma_0+\gamma^*)(x+\overline{u})^h}{1+c(x+\overline{u})^h}+1-r\right)-\overline{v}.
\end{aligned}\right.
\end{equation}

 Then the Jacobian of the system (\ref{bif2}) at the point (0,0) we have
\begin{equation}\label{jac2}
J(0, 0) = \begin{bmatrix}
\frac{(1 - \overline{u})(2 - h + 2c \overline{u}^h)}{1 + c \overline{u}^h} & -\frac{r}{\gamma_0} \\
\frac{(\gamma_0+\gamma^*) h (1 - \overline{u})}{1 + c \overline{u}^h}  & 1+\frac{\gamma^* \overline{u}^h}{1 + c \overline{u}^h}
\end{bmatrix}.
\end{equation}
and its characteristic equation is
$$\lambda^2-a(\gamma^*)\lambda+b(\gamma^*)=0,$$
where
$$a(\gamma^*)=Tr(J)=1+\frac{(1 - \overline{u})(2 - h + 2c \overline{u}^h)}{1 + c \overline{u}^h}+\frac{\gamma^* \overline{u}^h}{1 + c \overline{u}^h}=p(\overline{u})+\frac{\gamma^* \overline{u}^h}{1 + c \overline{u}^h},$$
 and since $\gamma_0$ is a solution of $q(\overline{u})=1,$ we get
\begin{align*}
 b(\gamma^*)&=\det(J)=1+\frac{\gamma^* (1 - \overline{u})}{1 + c \overline{u}^h}\left(\frac{\overline{u}^h(2 - h + 2c \overline{u}^h)}{1 + c \overline{u}^h}+\frac{rh}{\gamma_0}\right).
\end{align*}
Note that $a(0)=p(\overline{u})<2$ and $b(0)=1,$ so we have the following complex conjugate roots

\begin{equation}\label{bif5}
\lambda_{1,2}=\frac{1}{2}[a(\gamma^*)\pm i \sqrt{4b(\gamma^*)-a^2(\gamma^*)}].
\end{equation}
Thus,
\begin{equation}\label{bif6}
|\lambda_{1,2}|=\sqrt{b(\gamma^*)}
\end{equation}
and
\begin{equation}\label{bif7}
\frac{d|\lambda_{1,2}|}{d\gamma^*}\Bigm|_{\gamma^*=0}\!=\!\frac{\overline{u}(1-\overline{u})}{2\sqrt{b(\gamma^*)}(1+c\overline{u}^h)}\left(\frac{\overline{u}^h(2 - h + 2c \overline{u}^h)}{1 + c \overline{u}^h}+\frac{rh}{\gamma_0}\right)\Bigm|_{\gamma^*=0}\!>0.
\end{equation}

Hence, the transversality condition
\[
\frac{d|\lambda_{1,2}|}{d\gamma^*} \bigg|_{\gamma^* = 0} \neq 0
\]
has been established. Now, we consider the non-degeneracy condition, i.e., \( \lambda_{1,2}^m(0) \neq 1 \) for \( m = 1, 2, 3, 4 \). Recall that \( 1 < a(0) < 2 \) and \( b(0) = 1 \), from this we have that \( \lambda_{1,2}^m(0) \neq 1 \) for all \( m = 1, 2, 3, 4 \).
Therefore, all the conditions for the occurrence of a Neimark–Sacker bifurcation are satisfied.

\textbf{The third step}. To derive the normal form of the system (\ref{bif2}) when \( \gamma^* = 0 \), we expand the system (\ref{bif2}) into a Taylor series around \( (x, y) = (0, 0) \) up to third order.
\begin{equation}\label{bif8}
\left\{\begin{aligned}
&x^{(1)}=a_{10}x\!+\!a_{01}y\!+\!a_{20}x^2\!+\!a_{11}xy\!+\!a_{02}y^2\!+\!a_{30}x^3\!+\!a_{21}x^2y\!+\!a_{12}xy^2\!+\!a_{03}y^3\!+\!O(\rho^4)\\
&y^{(1)}= b_{10}x\!+\!b_{01}y\!+\!b_{20}x^2\!+\!b_{11}xy\!+\!b_{02}y^2\!+\!b_{30}x^3\!+\!b_{21}x^2y\!+\!b_{12}xy^2\!+\!b_{03}y^3\!+\!O(\rho^4),
\end{aligned}\right.
\end{equation}
where $\rho=\sqrt{x^2+y^2}$  and
\begin{equation}
\begin{split}
&a_{10}=\frac{(1 - \overline{u})(2 - h + 2c \overline{u}^h)}{1 + c \overline{u}^h}, \ \ a_{01}=-\frac{r}{\gamma_0}, \ \ a_{20}=-1+\frac{h(1-\overline{u})(1-h+c\overline{u}^h+hc\overline{u}^h)}{2\overline{u}(1+c\overline{u}^h)^2}, \\
& a_{11}=-\frac{h\overline{u}^{h-1}}{(1+c\overline{u}^h)^2},\ \ a_{02}=a_{03}=a_{12}=0,\\
&a_{30}=-\frac{h(1-\overline{u})\left(2(1+c\overline{u}^h)^2+3h(c^2\overline{u}^{2h}-1)+h^2(1-4c\overline{u}^h+c^2\overline{u}^{2h})\right)}{6\overline{u}^2(1+c\overline{u}^h)^3}, \\
&a_{21}=\frac{h\overline{u}^{h-2}(1-h+c\overline{u}^h+hc\overline{u}^h)}{2(1+c\overline{u}^h)^3}, \\
&b_{10}=\frac{\gamma_0 h (1 - \overline{u})}{1 + c \overline{u}^h}, \ \ b_{01}=1, \ \ b_{02}=b_{03}=b_{12}=0, \\
&b_{20}=-\frac{\gamma_0h(1-\overline{u})(1-h+c\overline{u}^h+hc\overline{u}^h)}{2\overline{u}(1+c\overline{u}^h)^2}, \ \ b_{11}=\frac{\gamma_0h\overline{u}^{h-1}}{(1+c\overline{u}^h)^2},\\
&b_{30}=\frac{\gamma_0h(1-\overline{u})\left(2(1+c\overline{u}^h)^2+3h(c^2\overline{u}^{2h}-1)+h^2(1-4c\overline{u}^h+c^2\overline{u}^{2h})\right)}{6\overline{u}^2(1+c\overline{u}^h)^3},\\
&b_{21}=-\frac{\gamma_0h\overline{u}^{h-2}(1-h+c\overline{u}^h+hc\overline{u}^h)}{2(1+c\overline{u}^h)^3}.
\end{split}
\end{equation}
Then
 \[
 J(\overline{E})=\begin{bmatrix}
a_{10} & ~~a_{01}\\
b_{10} &~~ b_{01}
\end{bmatrix}
\]
and eigenvalues of the Jacobian are:
$$\lambda_{1,2}=\frac{1+a_{10}\mp i\alpha}{2},$$
where $\alpha=\sqrt{3-a_{10}^2-2a_{10}}.$ Since $1<1+a_{10}=p(\overline{u})<2$ it is obvious that $3-a_{10}^2-2a_{10}>0.$

Using $q(\overline{u})=1,$ we observed that corresponding eigenvectors are non-zero and they are
$$v_{1,2}=\begin{bmatrix}-\frac{r}{2\gamma_0}\\1\end{bmatrix}\mp i\begin{bmatrix} \frac{r\alpha}{2\gamma_0(1-a_{10})} \\ 0\end{bmatrix}.$$

\textbf{The fourth step}. Now we find the normal form of the system (\ref{bif2}). We rewrite the system (\ref{bif8}) in the following form:
\begin{equation}\label{sf}
\mathbf{x}^{(1)}=J\cdot\mathbf{x}+H(\mathbf{x})
\end{equation}
where $\mathbf{x}=(x,y)^T$ and $H(\mathbf{x})$ is nonlinear part of the system (\ref{bif8}) without $O$ notion, i.e.,
\[
H(\mathbf{x})=\begin{bmatrix} a_{20}x^2+a_{11}xy+a_{30}x^3+a_{21}x^2y\\ b_{20}x^2+b_{11}xy+b_{30}x^3+b_{21}x^2y\end{bmatrix}
\]

Let matrix
\[
T= \begin{bmatrix} \frac{r\alpha}{2\gamma_0(1-a_{10})} & -\frac{r}{2\gamma_0}\\ 0 & 1\end{bmatrix}
\]
then
 \[
 T^{-1}= \begin{bmatrix} \frac{2\gamma_0(1-a_{10})}{r\alpha} & \frac{1-a_{10}}{\alpha} \\
 0 & 1\end{bmatrix}.
 \]

Let us denote $m=\frac{\alpha}{1-a_{10}},$ $n=\frac{r}{2\gamma_0}.$ Then we have
\[
T= \begin{bmatrix} mn & -n\\ 0 & 1\end{bmatrix}, \ \  T^{-1}= \begin{bmatrix} \frac{1}{mn} & \frac{1}{m} \\
 0 & 1\end{bmatrix}.
\]

By transformation, we get that
\[
\begin{bmatrix} x \\ y\end{bmatrix}=T\cdot \begin{bmatrix} X \\ Y\end{bmatrix}.
\]

the system (\ref{bif8}) transforms into the following system

\begin{equation}\label{sf1}
\mathbf{X}^{(1)}=T^{-1}\cdot J\cdot T\cdot\mathbf{X}+T^{-1}\cdot H(T\cdot\mathbf{x})+O(\rho_1)
\end{equation}
where $\mathbf{X}=(X,Y)^T$ and $\rho_1=\sqrt{X^2+Y^2}.$ Denote
\[
H(T\cdot\mathbf{x})=\begin{bmatrix} f(X,Y) \\ g(X,Y)\end{bmatrix}
\]
where
\[
\begin{aligned}
f(X,Y)=&a_{20}m^2n^2X^2+(a_{11}mn-a_{20}mn^2)XY+(a_{20}n^2-a_{11}n)Y^2+a_{30}m^3n^3X^3\\
&+(a_{21}m^2n^2-3a_{30}m^2n^3)X^2Y+(3a_{30}mn^3-2a_{21}mn^2)XY^2+(a_{21}n^2-a_{30}n^3)Y^3,
\end{aligned}
\]

\[
\begin{aligned}
g(X,Y)=&b_{20}m^2n^2X^2+(b_{11}mn-b_{20}mn^2)XY+(b_{20}n^2-b_{11}n)Y^2+b_{30}m^3n^3X^3\\
&+(b_{21}m^2n^2-3b_{30}m^2n^3)X^2Y+(3b_{30}mn^3-2b_{21}mn^2)XY^2+(b_{21}n^2-b_{30}n^3)Y^3,
\end{aligned}
\]

If we denote
\[
T^{-1}\cdot H(T\cdot\mathbf{x})=\begin{bmatrix} F(X,Y) \\ G(X,Y)\end{bmatrix}
\]
then we have
\begin{equation}
\begin{split}
&F(X,Y)=c_{20}X^2+c_{11}XY+c_{02}Y^2+c_{30}X^3+c_{21}X^2Y+c_{12}XY^2+c_{03}Y^3,\\
&G(X,Y)=d_{20}X^2+d_{11}XY+d_{02}Y^2+d_{30}X^3+d_{21}X^2Y+d_{12}XY^2+d_{03}Y^3.
\end{split}
\end{equation}

where
\begin{equation}
\begin{split}
&c_{20}=a_{20}mn+b_{20}mn^2, \ \ c_{11}=a_{11}-a_{20}n+b_{11}n-b_{20}n^2,\\
&c_{02}=\frac{a_{20}n-a_{11}+b_{20}n^2-b_{11}n}{m}, \ \ c_{30}= a_{30}m^2n^2+b_{30}m^2n^3,\\
&c_{21}=a_{21}mn-3a_{30}mn^2+b_{21}mn^2-3b_{30}mn^3, \\
& c_{12}= 3a_{30}n^2-2a_{21}n+3b_{30}n^3-2b_{21}n^2, \ \ c_{03}= \frac{a_{21}n-a_{30}n^2+b_{21}n^2-b_{30}n^3}{m},\\
&d_{20}=b_{20}m^2n^2,\ \ d_{11}=b_{11}mn-b_{20}mn^2,\ \ d_{02}=b_{20}n^2-b_{11}n,\ \ d_{30}=b_{30}m^3n^3,\\
&d_{21}=b_{21}m^2n^2-3b_{30}m^2n^3,\ \ d_{12}=3b_{30}mn^3-2b_{21}mn^2,\ \ d_{03}=b_{21}n^2-b_{30}n^3.\\
\end{split}
\end{equation}

Moreover, the partial derivatives at $(0,0)$ are
\begin{equation}
\begin{split}
&F_{XX}=2c_{20}, \ \ F_{XY}=c_{11}, \ \ F_{YY}=2c_{02}, \\
&F_{XXX}=6c_{30}, \ \ F_{XXY}=2c_{21}, \ \ F_{XYY}=2c_{12}, \ \ F_{YYY}=6c_{03}\\
&G_{XX}=2d_{20}, \ \ G_{XY}=d_{11}, \ \ G_{YY}=2d_{02}, \\
&G_{XXX}=6d_{30}, \ \ G_{XXY}=2d_{21}, \ \ G_{XYY}=2d_{12}, \ \ G_{YYY}=6d_{03}.
\end{split}
\end{equation}

\textbf{The fifth step}. Now, we have to compute the discriminating quantity \( \mathcal{L} \) using the following formula, which determines the stability of the invariant closed curve bifurcated from the Neimark-Sacker bifurcation of the system \eqref{sf1}:
\begin{equation}\label{lya}
\mathcal{L}=-Re\left[\frac{(1-2\lambda_1)\lambda_2^2}{1-\lambda_1}L_{11}L_{20}\right]-\frac{1}{2}|L_{11}|^2-|L_{02}|^2+Re(\lambda_2 L_{21}),
\end{equation}
where
\begin{equation}
\begin{split}
&L_{20}=\frac{1}{8}[(F_{XX}-F_{YY}+2G_{XY})+i(G_{XX}-G_{YY}-2F_{XY})],\\
&L_{11}=\frac{1}{4}[(F_{XX}+F_{YY})+i(G_{XX}+G_{YY})],\\
&L_{02}=\frac{1}{8}[(F_{XX}-F_{YY}-2G_{XY})+i(G_{XX}-G_{YY}+2F_{XY})],\\
&L_{21}=\frac{1}{16}[(F_{XXX}\!+\!F_{XYY}\!+\!G_{XXY}\!+\!G_{YYY})\!+\!i(G_{XXX}\!+\!G_{XYY}\!-\!F_{XXY}\!-\!F_{YYY})].
\end{split}
\end{equation}

To sum up the discussion above, we obtain at the following concluding theorem.

\begin{thm}\label{bifurcation} Let $\gamma>r(1+c)$ and let \( \gamma = \gamma_0 \) be a solution of \( q(\overline{u}) = 1 \).
If the parameter \( \gamma \) varies in a small neighborhood of \( \gamma_0 \) then the system (\ref{h12}) undergoes a Neimark-Sacker bifurcation at the fixed point \( \overline{E} = (\overline{u}, \overline{v}) \). Furthermore, if \( \mathcal{L} < 0 \) (respectively, \( \mathcal{L} > 0 \)), an attracting (respectively, repelling) invariant closed curve bifurcates from the fixed point for \( \gamma > \gamma_0 \) (respectively, \( \gamma <\gamma_0 \)).
\end{thm}

\subsection{Chaos control}

In this subsection, we will use the feedback control method for the system in order to
stabilize the chaotic orbit at an unstable fixed point of system (\ref{h12}). Towards this aim, by
adding feedback controlling force $\delta=-s_1(u - \overline{u})-s_2(v - \overline{v})$ with feedback gain
coefficients $s_1, s_2$ to the system (\ref{h12}), we obtain

\begin{equation}\label{h12c}
\begin{cases}
u^{(1)}=u(2-u)-\frac{u^hv}{1+cu^h}+\delta\\[2mm]
v^{(1)}=\frac{\gamma u^hv}{1+cu^h}+(1-r)v.
\end{cases}
\end{equation}

In addition, the Jacobian matrix of the controlled system (\ref{h12c}) is

\begin{equation}\label{jacchaos}
J^c(\overline{u}, \overline{v}) = \begin{bmatrix}
\frac{(1 - \overline{u})(2 - h + 2c \overline{u}^h)}{1 + c \overline{u}^h}-s_1 & -\frac{r}{\gamma}-s_2 \\
\frac{\gamma h (1 - \overline{u})}{1 + c \overline{u}^h}  & 1
\end{bmatrix}.
\end{equation}
whose characteristic polynomial is
\begin{equation}\label{chare}
\begin{aligned}
F(\lambda)=&\lambda^2-\left(1-s_1+\frac{(1 - \overline{u})(2 - h + 2c \overline{u}^h)}{1 + c \overline{u}^h}\right)\lambda+\frac{(1 - \overline{u})(2 - h + 2c \overline{u}^h)}{1 + c \overline{u}^h}-s_1\\
&+\frac{r h (1 - \overline{u})}{1 + c \overline{u}^h}+\frac{\gamma h (1 - \overline{u})}{1 + c \overline{u}^h}\cdot s_2.
\end{aligned}
\end{equation}

Let $\lambda_1$ and $\lambda_2$ are the roots of characteristic polynomial (\ref{chare}), then we have

\begin{equation}\label{sum}
\lambda_1+\lambda_2=-s_1+1+\frac{(1 - \overline{u})(2 - h + 2c \overline{u}^h)}{1 + c \overline{u}^h}
\end{equation}
and
\begin{equation}\label{dot}
\lambda_1\cdot\lambda_2=-s_1+\frac{(1 - \overline{u})(2 - h + 2c \overline{u}^h)}{1 + c \overline{u}^h}+\frac{r h (1 - \overline{u})}{1 + c \overline{u}^h}+\frac{\gamma h (1 - \overline{u})}{1 + c \overline{u}^h}\cdot s_2.
\end{equation}

\begin{lemma}\label{control} If the roots of the characteristic polynomial (\ref{chare}) satisfy $|\lambda_{1,2}|<1$ then (\ref{h12c}) is asymptotically stable.
\end{lemma}

\begin{proof} We take $\lambda_1=\pm1$ and $\lambda_1\lambda_2=1$ in order to obtain lines of marginal stability
for the system (\ref{h12c}). So, if $\lambda_1\lambda_2=1$, then from (\ref{dot}), we obtain

\begin{equation}
l_1:  s_1-\frac{\gamma h (1 - \overline{u})}{1 + c \overline{u}^h}\cdot s_2+1-\frac{(1 - \overline{u})(2 - h+rh + 2c \overline{u}^h)}{1 + c \overline{u}^h}=0.
\end{equation}

If $\lambda_1=1$ then from (\ref{sum}) and (\ref{dot}), we obtain

\begin{equation}
l_2:  r+\gamma s_2=0.
\end{equation}

Finally, if $\lambda_1=-1$ then from (\ref{sum}) and (\ref{dot}), we obtain

\begin{equation}
l_3: 2s_1-\frac{\gamma h (1 - \overline{u})}{1 + c \overline{u}^h}\cdot s_2-2-\frac{(1 - \overline{u})(4 -2 h +rh+ 4c \overline{u}^h)}{1 + c \overline{u}^h}=0.
\end{equation}
Therefore, the lines $l_1, l_2$ and $l_3$ explain the conditions for the eigenvalues satisfying
$|\lambda_{1,2}|<1$. Moreover, triangular region restricted by the lines $l_1, l_2$ and $l_3$ contains
stable eigenvalues. In Fig. \ref{reg}, we represent two triangular regions on the plane with the exact parameter values for operators (\ref{h1}) and (\ref{h2}).
\end{proof}

\begin{figure}[h!]
    \centering
    \subfigure[For $h=1$]{\includegraphics[width=0.45\textwidth]{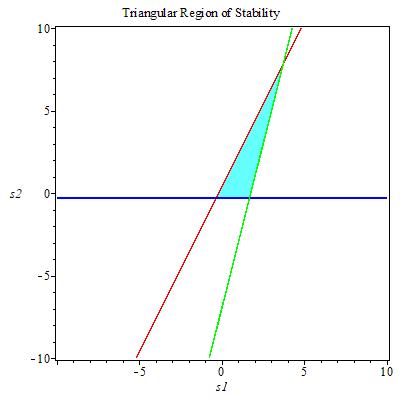}} \hspace{0.3in}
    \subfigure[For $h=2$]{\includegraphics[width=0.45\textwidth]{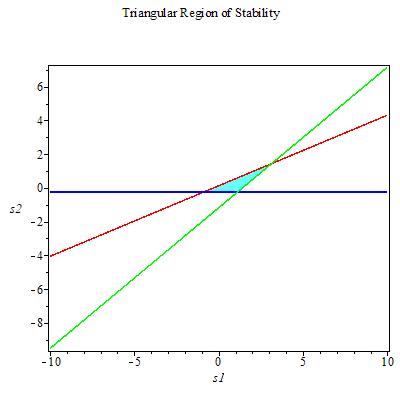}}
    \caption{Triangular regions of stability of the system  (\ref{h12c}) with parameters $r=0.5, \ \ c=1$ and $\gamma=2.$}
    \label{reg}
\end{figure}

\section{Global dynamics of (\ref{h12})}
 Let $h=1.$  Then operator (\ref{h12})  has the form
\begin{equation}\label{h1}
\begin{cases}
u^{(1)}=u(2-u)-\frac{uv}{1+cu}\\[2mm]
v^{(1)}=\frac{\gamma uv}{1+cu}+(1-r)v.
\end{cases}
\end{equation}

The positive fixed point $\overline{E}=(\overline{u}, \overline{v})$ has coordinates
$\overline{u}=\frac{r}{\gamma-rc},$  $\overline{v}=(1-\overline{u})(1+c\overline{u}).$

By solving the equation $q(\overline{u})=1$ with respect to $\gamma$ we get

\begin{equation}\label{gamma0}
\gamma=\frac{1-c+r+2rc+\sqrt{(1-c)^2+2r+6rc+r^2}}{2}=\gamma_0.
\end{equation}

Then the fixed point $\overline{E}$ remains attracting when $r(1+c)<\gamma<\gamma_0$ and it is repelling when $\gamma>\gamma_0.$ In Fig .\ref{sr} it is represented the stability region of the positive fixed point in the parameter space.

\begin{figure}[h!]
    \centering
    \subfigure[]{\includegraphics[width=0.45\textwidth]{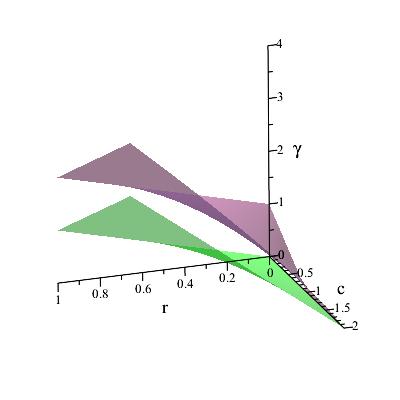}} \hspace{0.3in}
    \subfigure[]{\includegraphics[width=0.45\textwidth]{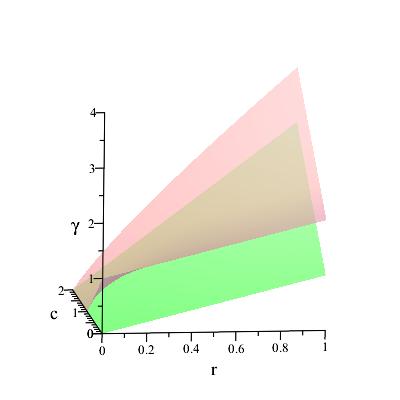}}
    \caption{Stability region between two surfaces (red and green) of the positive fixed point $\overline{E}$ for the operator (\ref{h1}), shown in the parameter space $(r, c,\gamma)$ with $0 < r \leq 1$ and $0 < c \leq 2$.}
    \label{sr}
\end{figure}

Let $h=2.$  Then operator (\ref{h12})  has the form
\begin{equation}\label{h2}
\begin{cases}
u^{(1)}=u(2-u)-\frac{u^2v}{1+cu^2}\\[2mm]
v^{(1)}=\frac{\gamma u^2v}{1+cu^2}+(1-r)v.
\end{cases}
\end{equation}
The positive fixed point $\overline{E}=(\overline{u}, \overline{v})$ has coordinates
$\overline{u}=\sqrt{\frac{r}{\gamma-rc}},$  $\overline{v}=\frac{(1-\overline{u})(1+c\overline{u}^2)}{\overline{u}}.$ From $q(\overline{u})=1,$ we obtain the following cubic equation with respect to $\gamma:$

\begin{equation}\label{gamma}
(2r-1)^2\gamma^3-\left(rc(2r-1)(6r-5)+4r^3\right)\gamma^2+4r^2c(r-1)(3rc-2c+2r)\gamma-4r^3c^2(r-1)^2(c+1)=0.
\end{equation}

It is evident that the set
\[
\mathcal{U} = \{(u,v) \in \mathbb{R}_+^2 \mid 0 \leq u \leq 2, v = 0\}
\]
is invariant with respect to the operators (\ref{h1}) and (\ref{h2}).

Furthermore, if \( 0 < r \leq 1 \), then the set
\[
\mathcal{V} = \{(u,v) \in \mathbb{R}_+^2 \mid u = 0, v \geq 0\}
\]
is also invariant under these operators.

Moreover, through straightforward calculations, it can be shown that a trajectory originating from the set \( \mathcal{U} \) converges to the fixed point \( E_1 = (1,0) \), while a trajectory starting from the set \( \mathcal{V} \) converges to the fixed point \( E_0 = (0,0) \).

\begin{pro}\label{prop1} Let \( v^{(1)} \) be defined as in (\ref{h12}).
If one of the following conditions on the parameters holds, then \( v^{(1)} \geq 0 \) for any \( u \in [0,1] \) and $v\geq0$:

(a) \( \gamma \leq -1 \), \( c \geq -1 - \gamma \), and \( 0 < r \leq \frac{c+1+\gamma}{c+1} \);

(b) \( -1 < \gamma \leq 0 \), \( c > 0 \), and \( 0 < r \leq \frac{c+1+\gamma}{c+1} \);

(c) \( \gamma > 0 \), \( c > 0 \), and \( 0 < r \leq 1 \).

\end{pro}
\begin{proof}
The proof follows directly from the condition \( v^{(1)} \geq 0 \), which is equivalent to
\[
u^h (\gamma + c - r c) + 1 - r \geq 0.
\]
\end{proof}

\begin{lemma}\label{invm}
Assume that one of the conditions \((a)-(c)\) is satisfied and \( \gamma \leq r(1+c) \). Then the following sets are invariant with respect to the operator (\ref{h1}):

\begin{itemize}
    \item[(i)]
    If \( c \leq \frac{1}{2} \), then
    \[
    M_1 = \left\{ (u,v) \in \mathbb{R}^2 \mid 0 \leq u \leq 1, \, 0 \leq v \leq (2 - u)(1 + c u) \right\}.
    \]

    \item[(ii)] If \( c \geq 1 \), then
    \[
    M_2 = \left\{ (u,v) \in \mathbb{R}^2 \mid 0 \leq u \leq 1, \, 0 \leq v \leq 2 \right\}.
    \]

    \item[(iii)] If \( \frac{1}{2} < c < 1 \), then
    \[
    M_3 = \left\{ (u,v) \in \mathbb{R}^2 \mid 0 \leq u \leq 1, \, 0 \leq v \leq \min \{ 2, (2 - u)(1 + c u) \} \right\}.
    \]
\end{itemize}

\end{lemma}

\begin{proof}
Let \( (u,v) \in M_i \) for \( i = 1,2,3 \). We prove that \( (u^{(1)},v^{(1)}) \in M_i \). Since \( v \leq (2-u)(1+cu) \) for all \( u \in [0,1] \), we have \( u^{(1)} \geq 0 \).

Moreover, the maximum value of the function \( x(2-x) \) is 1, so
\[
u^{(1)} = u(2-u) - \frac{uv}{1+cu} \leq u(2-u) \leq 1.
\]
Thus, we obtain \( 0 \leq u^{(1)} \leq 1 \).

According to Proposition \ref{prop1}, it is clear that \( v^{(1)} \geq 0 \) for all \( u \geq 0 \). Moreover, from \( \gamma \leq r(1+c) \), we get \( \frac{r}{\gamma - rc} \geq 1 \geq u \), which implies \( v^{(1)} \leq v \) for all \( u \in [0,1] \).

Now, we need to prove the only remaining condition:
\[
v^{(1)} \leq (2 - u^{(1)})(1 + c u^{(1)}).
\]

If \( v \leq (1-u)(1+cu) \), then we have \( u^{(1)} \geq u \), and since \( v^{(1)} \leq v \), it follows that
\[
v^{(1)} \leq (2 - u^{(1)})(1 + c u^{(1)})
\]
is obvious. So, we consider the case \( v > (1-u)(1+cu) \), i.e., \( u^{(1)} < u \).

 \textbf{Case} (i): \( c \leq \frac{1}{2} \)
The function \( f(x) = (2-x)(1+cx) \) is decreasing on the interval \( [0,1] \). Therefore,
\[
f(u^{(1)}) \geq f(u),
\]
which implies
\[
v^{(1)} \leq v \leq (2 - u)(1 + c u) \leq (2 - u^{(1)})(1 + c u^{(1)}).
\]
The form of \( M_1 \) is given in Fig. \ref{image1}.

 \textbf{Case }(ii): \( c \geq 1 \)
Consider the equation \( f(x) = 2 \). The positive solution is
\[
x_0 = \frac{2c - 1}{c}.
\]
Since \( c \geq 1 \), we have \( x_0 \geq 1 \). Thus, the rectangle \( M_2 \) is located inside \( M_1 \) (Fig \ref{image2}), which implies
\[
v^{(1)} \leq v \leq 2 \leq (2 - u^{(1)})(1 + c u^{(1)})
\]
in \( M_2 \).

 \textbf{Case} (iii): \( \frac{1}{2} < c < 1 \)
In this case, \( x_0 < 1 \). Since \( v^{(1)} < v \) and \( u^{(1)} < u \), the point \( (u^{(1)}, v^{(1)}) \) remains in the set \( M_3 \) (Fig \ref{image3}).

Thus, the proof is complete.
\end{proof}

\begin{figure}[h]
    \centering
    \subfigure[The set $M_1$ with $c=0.4$]{
        \includegraphics[width=0.3\textwidth]{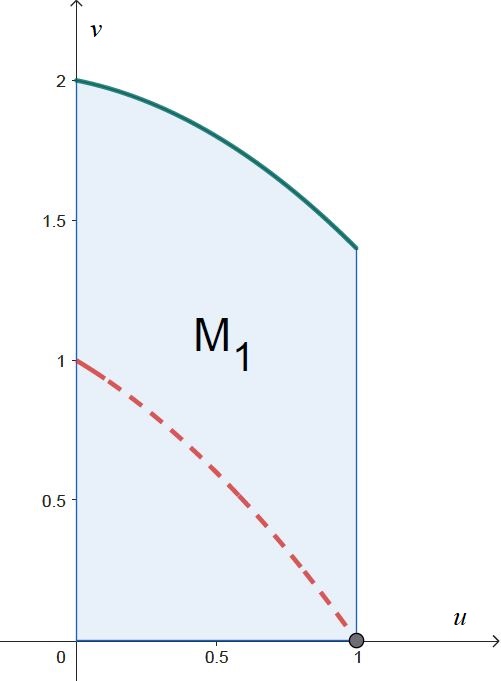}
        \label{image1}
    }
    \subfigure[The set $M_2$ with $c=1.5$]{
        \includegraphics[width=0.25\textwidth]{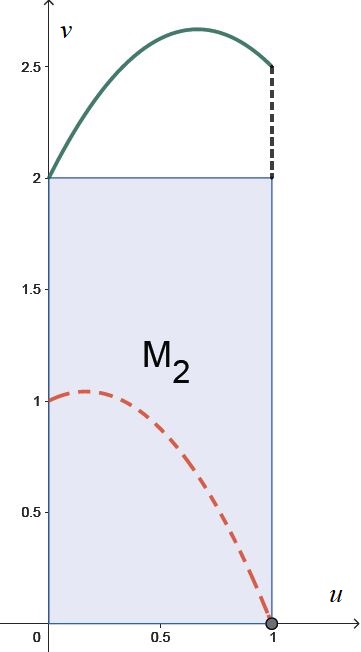}
        \label{image2}
    }
    \subfigure[The set $M_3$ with $c=0.75$]{
        \includegraphics[width=0.25\textwidth]{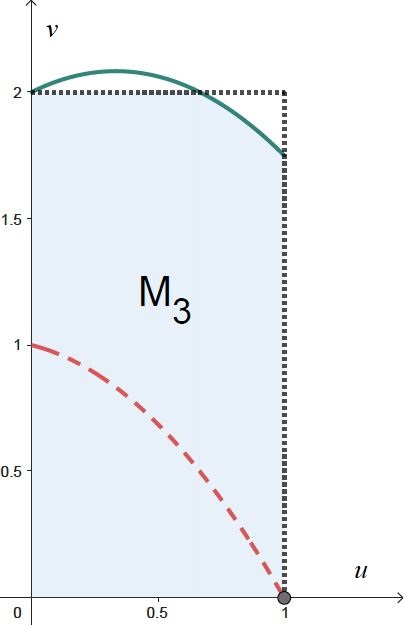}
        \label{image3}
    }
    \caption{Three invariant sets}
    \label{invsets}
\end{figure}

\begin{rk} If $c>1/2$ then the set $M_2$ is not invariant w.r.t. operator (\ref{h1}). For example, if $\gamma=10, c=19, r=0.5,$ then the point $(u,v)=(0.1; 4.62)\in M_2,$ while $(u^{(1)},v^{(1)})\approx(0.03; 3.9)\notin M_2$ which means $v^{(2)}<0.$
\end{rk}

\begin{pro}\label{prop2}
Assume that one of the conditions \((a)-(c)\) is satisfied and \( \gamma \leq r(1+c) \). Then, for any initial point \( (u^0,v^0) \in M_i \), where \( i = 1,2,3 \) and $u^0>0$, the trajectory of the operator (\ref{h1}) converges to the fixed point \( E_1 = (1,0) \).
\end{pro}

\begin{proof}
Define the following function on the set \( M_i \):
\[
L(u,v) = v.
\]
Clearly, \( L(u,v) \geq 0 \) and \( L(1,0) = 0 \). Consider the difference:
\[
\Delta L = L(u^{(1)},v^{(1)}) - L(u,v) = v^{(1)} - v \leq 0.
\]
The set where \( \Delta L = 0 \) is given by:
\[
\mathcal{B} = \{(x,y) \in M_i : \Delta L = 0\} = \{(u,0)\}.
\]
Assume \( u > 0 \) (the case \( u = 0 \) was mentioned earlier). Then, the largest invariant set in \( \mathcal{B} \) is the fixed point \( E_1 = (1, 0) \).

Since the sets \( M_i \) are compact, it follows from LaSalle's Invariance Principle that any trajectory starting in \( M_i \) converges to the fixed point \( E_1 \).
\end{proof}

Define the function
\begin{equation}
\psi(x)=\frac{(2 - x)(1 + c x^2)}{x}
\end{equation}

\begin{lemma}\label{invn}
Assume that one of the conditions \((a)-(c)\) is satisfied and \( \gamma \leq r(1+c) \). Then the following sets are invariant with respect to the operator (\ref{h2}):

\begin{itemize}
    \item[(i)]
    If \( c \leq \frac{27}{4} \), then
    \[
    N_1 = \left\{ (u,v) \in \mathbb{R}^2 \mid 0 \leq u \leq 1, \, 0 \leq v \leq \frac{(2 - u)(1 + c u^2)}{u} \right\}.
    \]

    \item[(ii)] If \( c > \frac{27}{4} \), then
    \[
    N_2 = \left\{ (u,v) \in \mathbb{R}^2 \mid 0 \leq u \leq 1, \, 0 \leq v \leq \psi_{min} \right\}.
    \]

  \end{itemize}

\end{lemma}

\begin{proof} Let $(u,v)\in N_{i}, i=1,2.$ As shown in the proof of Lemma \ref{invm}, $0\leq u^{(1)}\leq1$ and $v^{(1)\geq0}$ is obvious. We have to prove that $v^{(1)} \leq \frac{(2 - u^{(1)})(1 + c (u^{(1)})^2)}{u^{(1)}}$ (case (i))and $v^{(1)} \leq \psi_{min}$ (case (ii)). Here also the condition \( \gamma \leq r(1+c) \) supplies that \( v^{(1)} \leq v \). Moreover, if \( v \leq \frac{(1-u)(1+cu)}{u} \), then we have \( u^{(1)} \geq u \) and
\[
v^{(1)} \leq \frac{(2 - u^{(1)})(1 + c u^{(1)})}{u^{(1)}}
\]
follows directly. So, we consider the case \( v >\frac{ (1-u)(1+cu)}{u} \), i.e., \( u^{(1)} < u \).

\textbf{Case} (i): \( c \leq \frac{27}{4} \)

The derivative of the function $\psi(x)$ is $\psi'(x)=-\frac{2(cx^3-cx^2+1)}{x^2}$ and $\psi'(x)\leq0$ if and only if $cx^3-cx^2+1\geq0.$  Denote $\nu(x)=cx^3-cx^2+1.$ Then $\nu(0)=\nu(1)=1$ and $\nu(x)$ has its minimum value at $x_0=\frac{2}{3}.$ So, $\nu(x_0)=\frac{27-4c}{27}.$  Since \( c \leq \frac{27}{4} \) we have that $cx^3-cx^2+1\geq0,$ so the function $\psi(x)$ is decreasing. Thus, $\psi(u^{(1)})\geq\psi(u)$ which we arrive
\[
v^{(1)} \leq v \leq \frac{(2 - u)(1 + c u)}{u} \leq \frac{(2 - u^{(1)})(1 + c u^{(1)})}{u^{(1)}}.
\]

\textbf{Case} (ii): \( c > \frac{27}{4} \)

Then since $\nu(x_0)=\frac{27-4c}{27}<0$ the function $\nu(x)$ has two zeros $x_{min}\in(0,2/3)$ and $x_{max}\in(2/3,1).$ At the point $x_{min}$, the function $\psi(x)$ has a minimum value $\psi_{min}.$ So, from the structure of the set $N_2$ we get the proof of the lemma.
\end{proof}

\begin{pro}\label{prop3}
Assume that one of the conditions \((a)-(c)\) is satisfied and \( \gamma \leq r(1+c) \). Then, for any initial point \( (u^0,v^0) \in N_i \), where \( i = 1,2 \) and $u^0>0$, the trajectory of the operator (\ref{h2}) converges to the fixed point \( E_1 = (1,0) \).
\end{pro}

\begin{proof} The proof is very similar to the proof of Proposition \ref{prop2}.
\end{proof}

\begin{rk}
If \( \gamma > r(1+c) \), i.e., a positive fixed point exists, then the behavior of the trajectory varies depending on the parameters and initial conditions. In Fig. \ref{figh12}, we illustrate several cases where the trajectory converges to the positive fixed point.
\end{rk}

\begin{figure}[h!]
    \centering
    \subfigure[\tiny$r=0.5, c=1, \gamma= 1.2, \overline{E}\approx(0.71,0.48), (u^0, v^0)=(0.8, 1.5).$]{\includegraphics[width=0.45\textwidth]{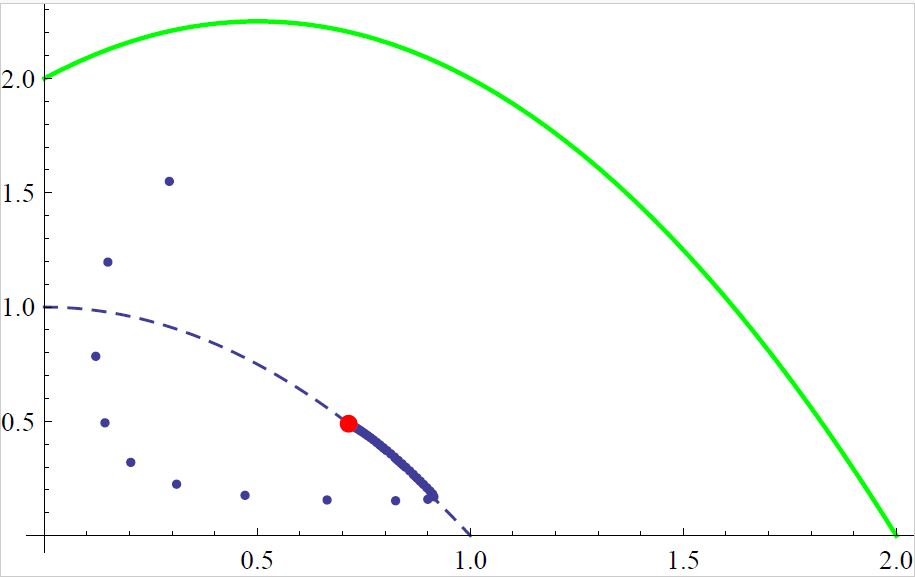}} \hspace{0.3in}
    \subfigure[\tiny$r=0.5, c=1, \gamma= 1.6, \overline{E}\approx(0.45,0.79), (u^0, v^0)=(0.8, 1.5).$]{\includegraphics[width=0.45\textwidth]{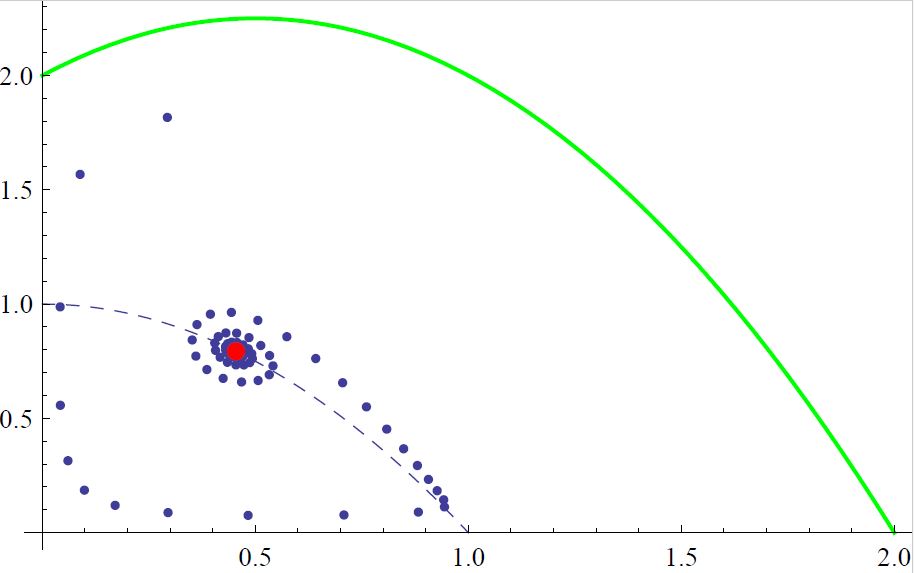}}
    \subfigure[\tiny$r=0.8, c=2, \gamma= 3, \overline{E}\approx(0.75,0.69), (u^0, v^0)=(0.86, 2.4).$]{\includegraphics[width=0.45\textwidth]{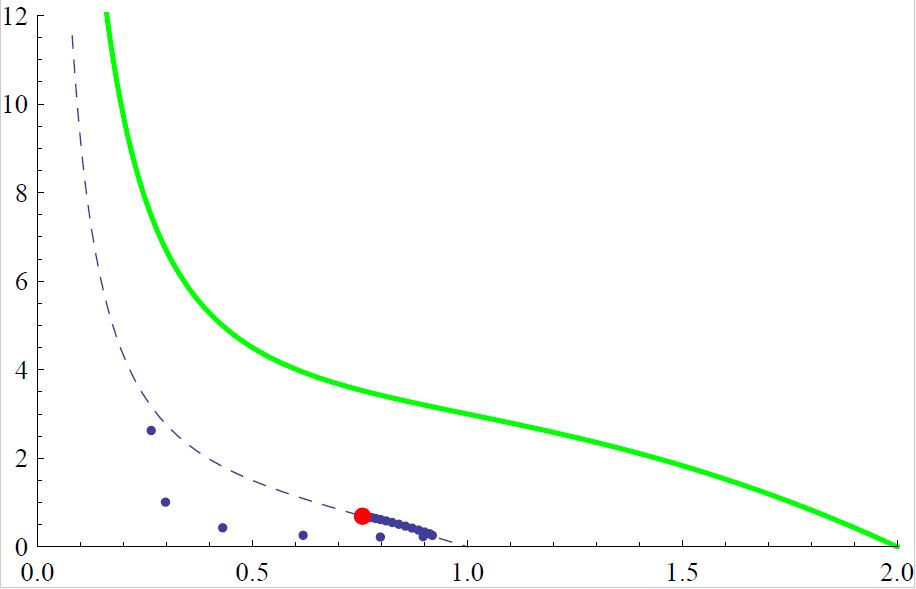}} \hspace{0.3in}
    \subfigure[\tiny$r=0.8, c=12, \gamma= 12, \overline{E}\approx(0.57,3.66), (u^0, v^0)=(0.6, 8).$]{\includegraphics[width=0.45\textwidth]{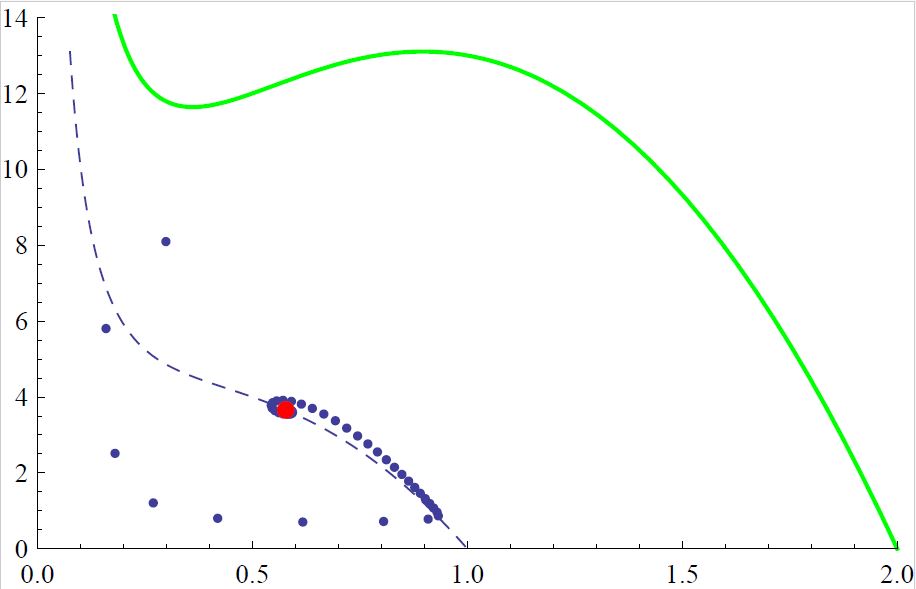}}
 \caption{In Figures (a) and (b), the trajectories are shown for system (\ref{h1}) (i.e., \( h = 1 \)), while in Figures (c) and (d), the trajectories are presented for system (\ref{h2}) (i.e., \( h = 2 \)).}
    \label{figh12}
\end{figure}

\section{numerical simulations}

\subsection{Neimark-Sacker bifurcation of the system (\ref{h1})} Consider the system (\ref{h1}) with parameters \( c = 1 \) and \( r = 0.5 \). Then the equation \( q(\overline{u}) = 1 \) has a solution
\[
\gamma_0=\frac{3+\sqrt{17}}{4}\approx 1.78>r(1+c);
\]
and the fixed point is

\[
\overline{E} \approx (0.3906, 0.8474).
\]

The corresponding multipliers are

\[
\lambda_1 \approx 0.8902 - 0.4554i, \quad \lambda_2 \approx 0.8902 + 0.4554i.
\]

The relevant coefficients are computed as

\[
L_{20} \approx 0.0174 + 0.0575i, \quad L_{11} \approx -0.1887 - 0.1654i,
\]
\[
L_{02} \approx -0.2738 - 0.0827i, \quad L_{21} \approx -0.0733 + 0.0158i.
\]

Additionally, we obtain

\[
\mathcal{L} \approx -0.2142 < 0.
\]

Thus, according to Theorem \ref{bifurcation}, an attracting invariant closed curve bifurcates from the fixed point for \( \gamma^* > 0 \), and the system (\ref{h1}) undergoes a Neimark–Sacker bifurcation. In Fig. \ref{diag}, the Neimark–Sacker bifurcation diagrams and the corresponding Maximum Lyapunov exponents for system (\ref{h1}) are presented.

In Figure \ref{fig1} (a) the fixed point remains attracting since \( q(\overline{u}) < 1 \). In Figures \ref{fig1} (b) and (c), an invariant closed curve is formed as \( q(\overline{u}) > 1 \), and moreover, this curve is attracting. In Figures \ref{fig1} (d), (e), and (f), we present trajectories for various values of \( \gamma \) with fixed parameters \( r = 0.5 \) and \( c = 1 \).

\subsection{Neimark-Sacker bifurcation of the system (\ref{h2})} Consider the system (\ref{h2}) with parameters \( c = 2 \) and \( r = 0.8 \). Note that the bifurcation parameter $\gamma_0$ is a positive solution of equation (\ref{gamma}).

Then the equation (\ref{gamma}) has a solution
\[
\gamma_0\approx 6.3
\]
and the positive fixed point is

\[
\overline{E} \approx (0.4125, 1.9085).
\]

The corresponding multipliers are

\[
\lambda_1 \approx 0.6491 - 0.7606i, \quad \lambda_2 \approx 0.6491 + 0.7606i.
\]

The relevant coefficients are computed as

\[
L_{20} \approx 0.0412 + 0.0987i, \quad L_{11} \approx -0.0190 - 0.0930i,
\]
\[
L_{02} \approx -0.1583 - 0.0076i, \quad L_{21} \approx 0.0108 - 0.0065i.
\]

Additionally, we obtain

\[
\mathcal{L} \approx -0.0141 < 0.
\]

Thus, according to Theorem \ref{bifurcation}, an attracting invariant closed curve bifurcates from the fixed point for \( \gamma^* > 0 \).

In Fig. \ref{diag2}, the Neimark–Sacker bifurcation diagrams and the corresponding Maximum Lyapunov exponents for system (\ref{h2}) are presented.

In Figure \ref{fig2} (a) the fixed point remains attracting since \( q(\overline{u}) < 1 \). In Figures \ref{fig2} (b) and (c), an invariant closed curve is formed as \( q(\overline{u}) > 1 \), and moreover, this curve is attracting. In Figures \ref{fig2} (d), (e), and (f), we present trajectories for various values of \( \gamma \) with fixed parameters \( r = 0.8 \) and \( c = 2 \).

\section{Conclusion}

The dynamics of plankton proliferation are inherently complex, and constructing comprehensive models that accurately capture these dynamics remains a significant challenge in ecological modeling. Extensive research has been devoted to analyzing both continuous-time and discrete-time systems, focusing on key dynamical properties such as the stability of equilibria, the existence and uniqueness of limit cycles, and the emergence of bifurcations.

In the context of continuous-time models similar to the one studied here, \cite{Chen2} established the global attractivity of the positive equilibrium and demonstrated the existence of a limit cycle. Similarly, \cite{Peng} investigated the conditions under which bifurcations occur at the positive equilibrium, contributing to the understanding of transitions in population dynamics.

By contrast, discrete-time models have received comparatively less attention in the literature, despite their relevance for seasonal environments or data-driven ecological systems. For instance, in \cite{SH}, the authors analyzed a discrete-time model incorporating a Holling Type II functional response \( f(P) = \frac{P}{1 + cP} \) and a linear toxin release function \( g(P) = P \). They examined the existence and nature of positive fixed points and demonstrated the occurrence of a Neimark--Sacker bifurcation. In a related study, Elettreby~\cite{Ele} considered a discrete Lotka--Volterra predator-prey model with a combination of Holling Type I and Type III responses, revealing the presence of both Neimark--Sacker and flip bifurcations at the interior equilibrium.

These prior works underscore the rich dynamical behavior exhibited by plankton systems and highlight the need for continued exploration of discrete-time models, particularly those incorporating nonlinear interactions and toxicity effects.

In this study, we analyzed the dynamics of a discrete-time phytoplankton-zooplankton model in which both the predator's functional response and the distribution of toxin substances are described by a common Holling-type expression. This unified framework enabled a systematic comparison between Holling Type II and Type III interactions and offered deeper insights into their ecological effects.

We established that the system undergoes a Neimark--Sacker bifurcation at a unique positive fixed point \( \overline{E} = (\overline{u}, \overline{v}) \) when the predation parameter \( \gamma \) crosses a critical value \( \gamma_0 \) (Theorem~\ref{bifurcation}). This bifurcation leads to the emergence of an attracting or repelling invariant closed curve depending on the sign of the first Lyapunov coefficient \( \mathcal{L} \). From an ecological perspective, this behavior corresponds to the transition from stable population equilibria to persistent oscillations or quasi-periodic dynamics, a phenomenon frequently observed in natural plankton communities.

Moreover, it was shown that when \( \gamma \leq r(1+c) \), all trajectories originating within biologically meaningful invariant sets converge to the boundary fixed point \( E_1 = (1,0) \) (Propositions~\ref{prop2} and~\ref{prop3}). This result highlights the possibility of zooplankton extinction under certain environmental or toxicological stress conditions, with important implications for understanding ecosystem resilience and collapse scenarios.

To address the instability and chaotic behavior identified near the fixed point, we proposed a feedback control strategy by introducing a corrective term to the system dynamics. By suitably selecting feedback gain parameters, we proved that the controlled system becomes asymptotically stable (Lemma~\ref{control}). This theoretical development suggests that in real-world settings, management strategies such as regulating nutrient supply, mitigating pollutant inputs, or implementing dynamic harvesting policies could help stabilize plankton populations and maintain ecosystem balance.

While the present model yields valuable theoretical insights into the dynamics of plankton populations, its applicability to real marine ecosystems is constrained by several simplifying assumptions. In particular, the omission of spatial structure and the assumption of a constant toxin release rate may overlook critical ecological processes such as localized interactions, spatial dispersion, and adaptive behavioral responses. These simplifications highlight important directions for future research. We intend to extend our analysis by incorporating spatially explicit dynamics, modeling variable toxin release mechanisms, and examining a wider array of ecological scenarios. In addition, we aim to systematically explore all nine biologically meaningful cases identified in \cite{Chatt}, thereby enriching the model’s ecological relevance and generality.

Building upon this foundation, future studies will also investigate a broader class of functional responses and introduce additional trophic groups, such as mixoplankton and bacteria, which play vital roles in marine ecosystems. Furthermore, we will analyze the influence of environmental noise and spatial heterogeneity, both of which are essential for capturing the variability and complexity observed in natural settings. A key objective will be to establish rigorous analytical conditions for the onset of chaotic attractors and other complex dynamics. Through these extensions, we seek to develop a more comprehensive and predictive modeling framework to support the understanding, management, and conservation of ocean ecosystems under changing environmental conditions.

\section*{Declerations}

 \textbf{Ethical approval} The author confirms that this study does not involve human participants or animals, and no ethical approval was required.

\textbf{Funding} No funding was received for conducting this study.

\textbf{Conflict of interest} The author declares that there is no Conflict of interest regarding the publication of
this paper.

\textbf{Data availability} Not applicable.

\begin{figure}[h!]
    \centering
    \subfigure[]{\includegraphics[width=0.56\textwidth]{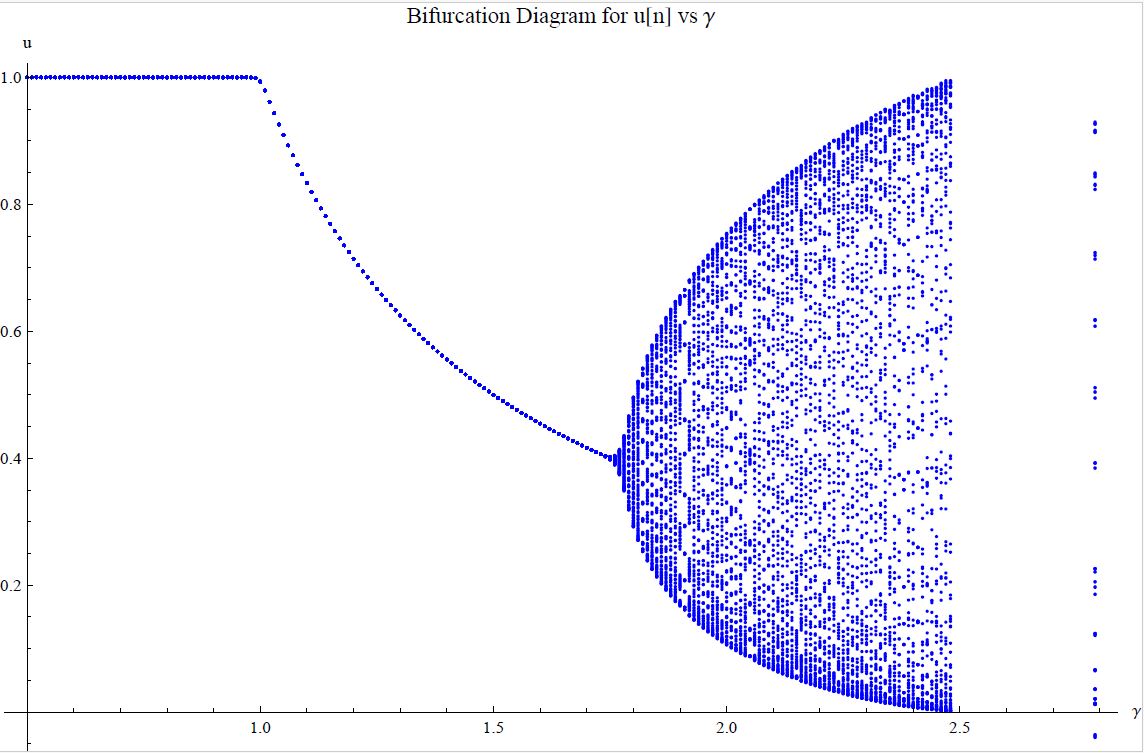}} \hspace{0.3in}
    \subfigure[]{\includegraphics[width=0.56\textwidth]{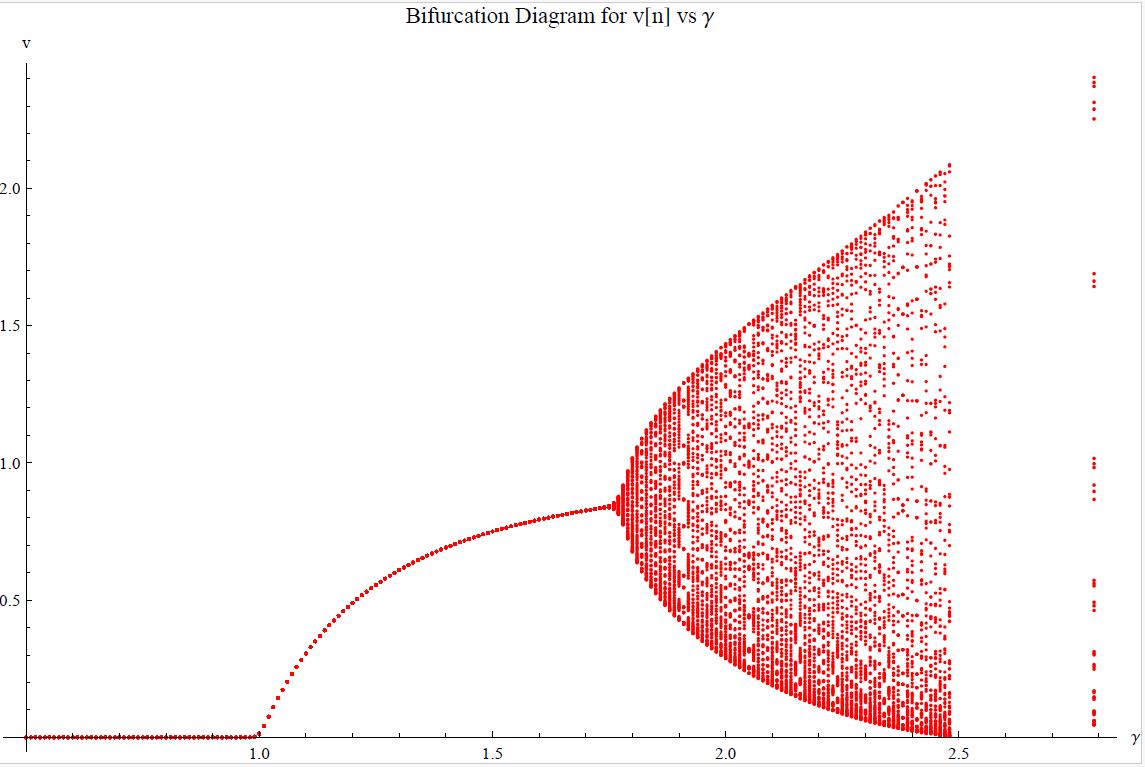}}
    \subfigure[]{\includegraphics[width=0.56\textwidth]{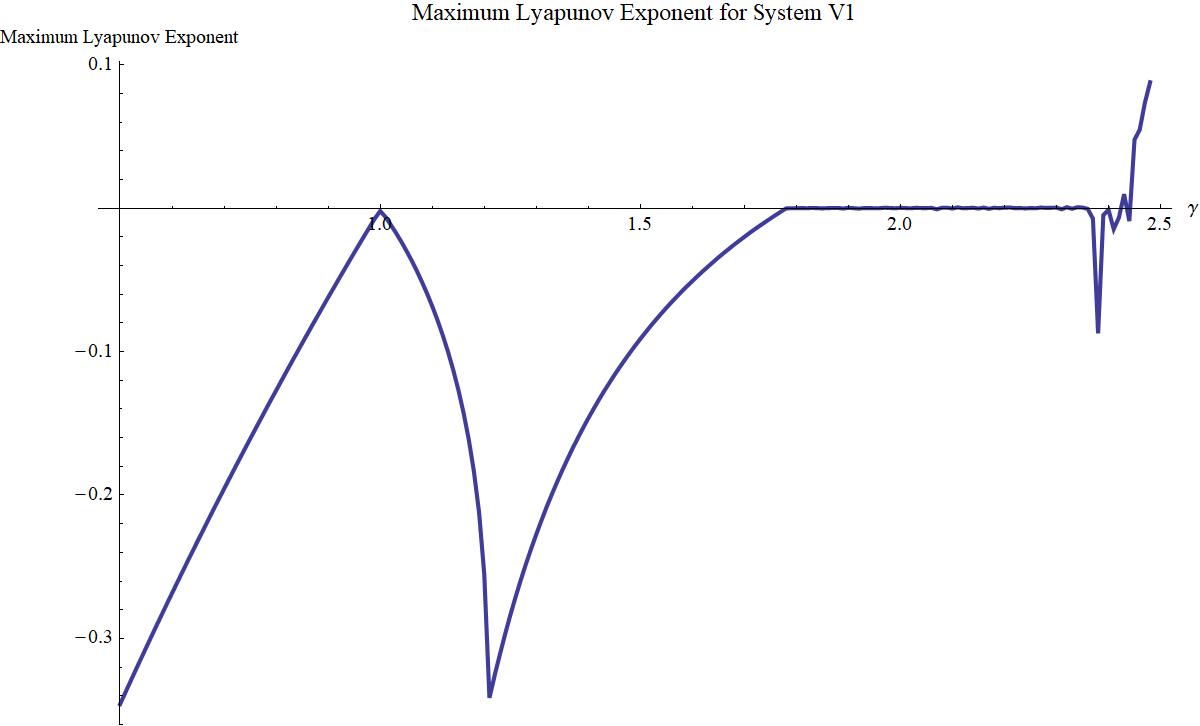}}
    \caption{\tiny Bifurcation diagrams for the system \ref{h1} with the parameters $r = 0.5, c =
1,$ and initial values $u^0 = 0.35, v^0 = 0.6$ when the bifurcation parameter $\gamma$ varying on the interval $0.5\leq\gamma\leq3$. In (c), the Maximum Lyapunov exponent corresponding to (a) and (b) are presented. The maximum Lyapunov exponent indicates that chaotic dynamics are observed when approximately \( 1.78 < \gamma < 2.357 \) and \( 2.443 < \gamma < 2.482 \).}
    \label{diag}
\end{figure}

\begin{figure}[h!]
    \centering
    \subfigure[\tiny$\gamma=1.775, \overline{u}\approx0.3921, q(\overline{u})\approx0.9996, (0.35, 0.6).$]{\includegraphics[width=0.45\textwidth]{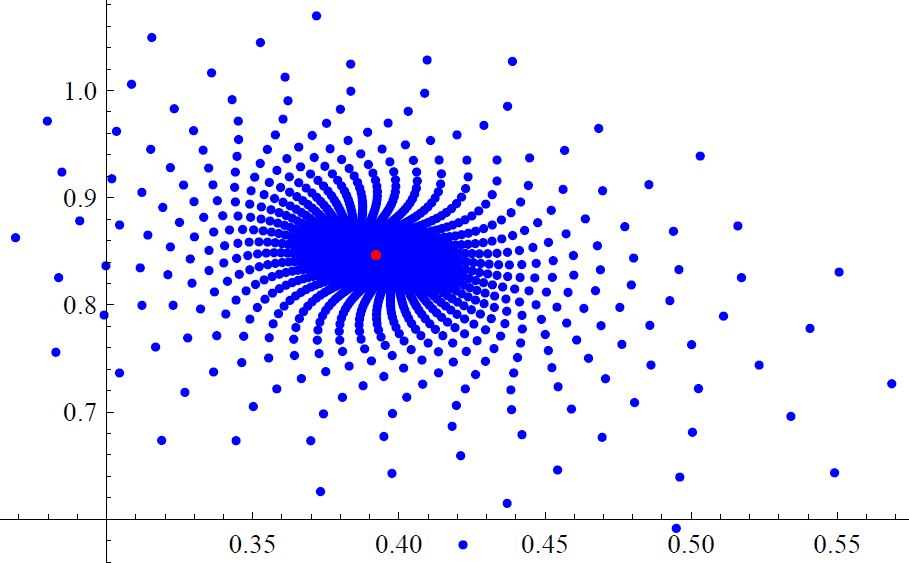}} \hspace{0.3in}
    \subfigure[\tiny$\gamma=1.79, \overline{u}\approx0.3875, q(\overline{u})\approx1.00414, (0.35, 0.7).$]{\includegraphics[width=0.45\textwidth]{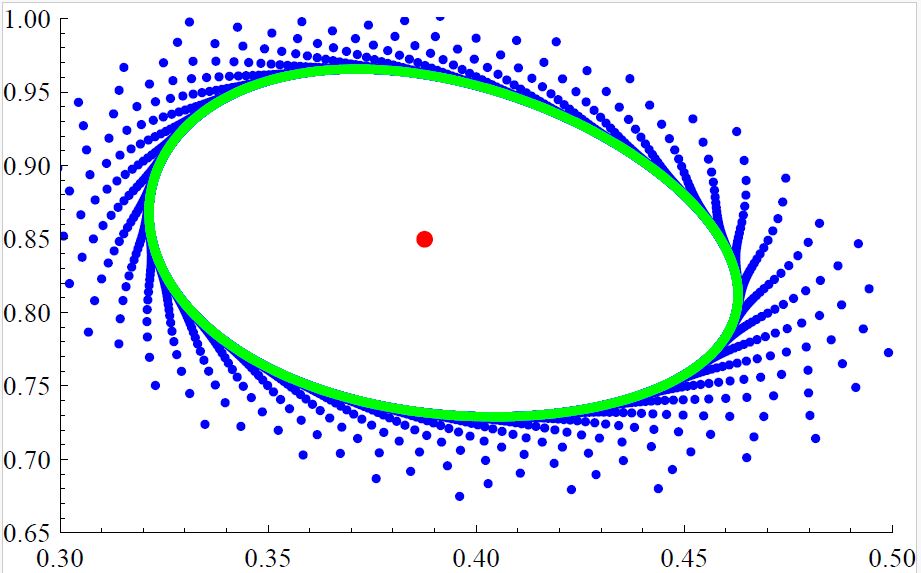}}
    \subfigure[\tiny$\gamma=1.79, \overline{u}\approx0.3875, q(\overline{u})\approx1.00414, (0.4, 0.8).$]{\includegraphics[width=0.45\textwidth]{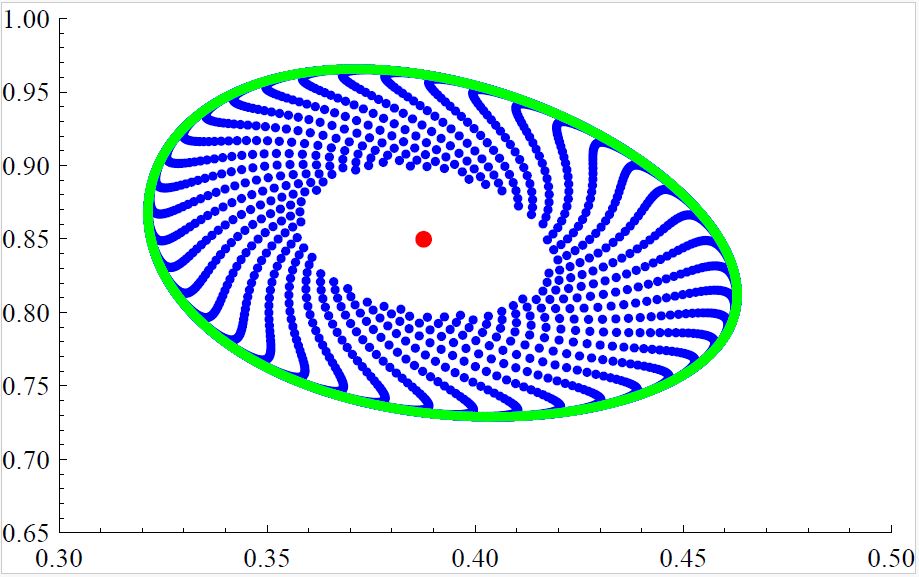}} \hspace{0.3in}
    \subfigure[\tiny$\gamma=2.2, (u^0, v^0)=(0.33, 0.96).$]{\includegraphics[width=0.45\textwidth]{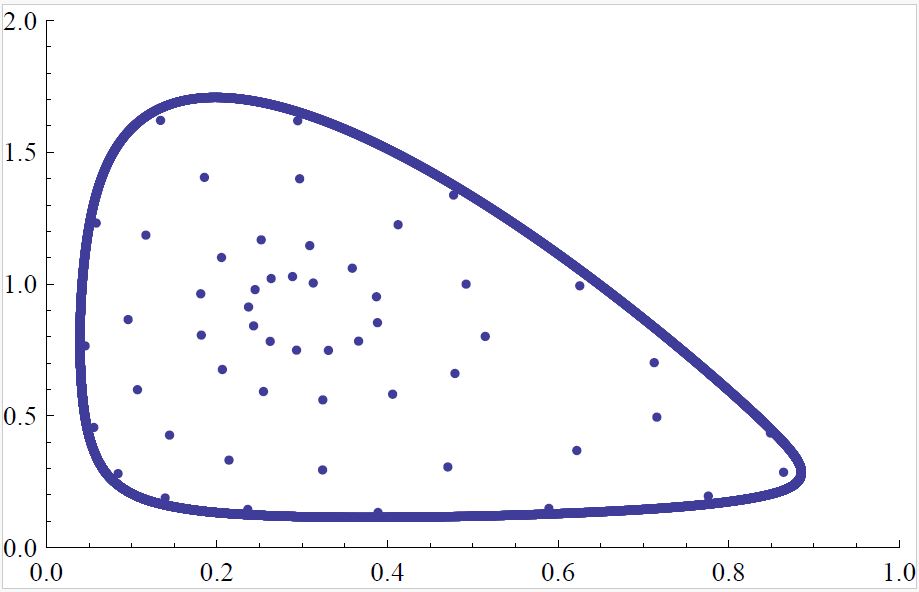}} \hspace{0.3in}
    \subfigure[\tiny$\gamma=2.43, (u^0, v^0)=(0.33, 0.96).$]{\includegraphics[width=0.45\textwidth]{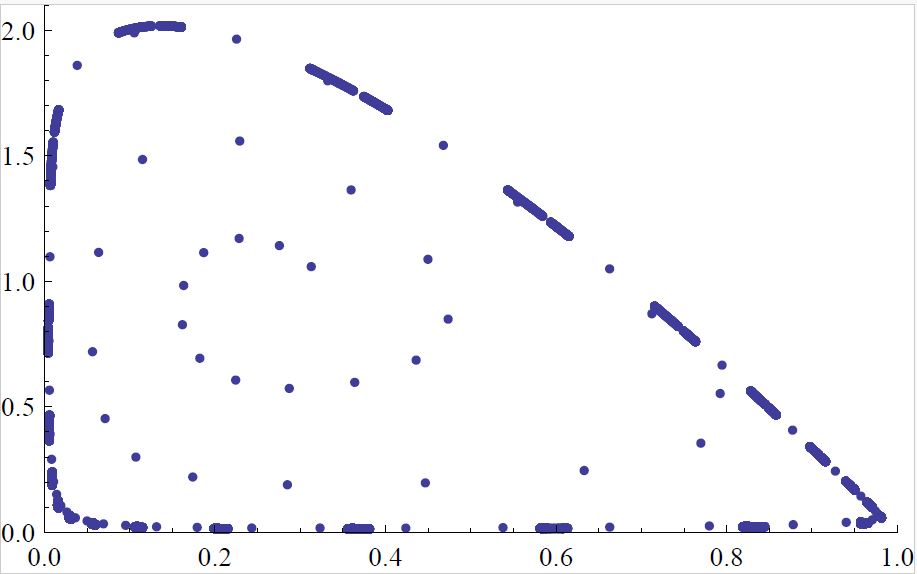}} \hspace{0.3in}
    \subfigure[\tiny$\gamma=2.48, (u^0, v^0)=(0.33, 0.96).$]{\includegraphics[width=0.45\textwidth]{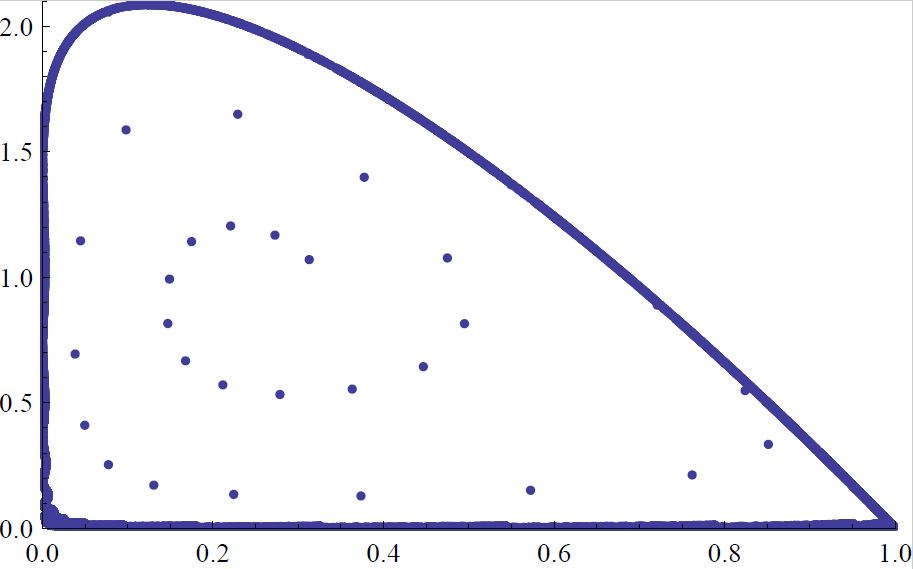}}
    \caption{\tiny Phase portraits for system (\ref{h1}) with parameters \( c = 1 \), \( r = 0.5 \), and \( n = 10,000 \). The red point represents the fixed point \( \overline{E} \), while the green curve denotes an attracting invariant closed curve.
In (a), \( q(\overline{u}) < 1 \), so the fixed point remains attracting. In (b), the initial point is chosen outside the invariant closed curve, whereas in (c), it is taken from inside. In (d), (e), and (f), we present trajectories for various values of \( \gamma \).}
    \label{fig1}
\end{figure}

\begin{figure}[h!]
    \centering
    \subfigure[]{\includegraphics[width=0.56\textwidth]{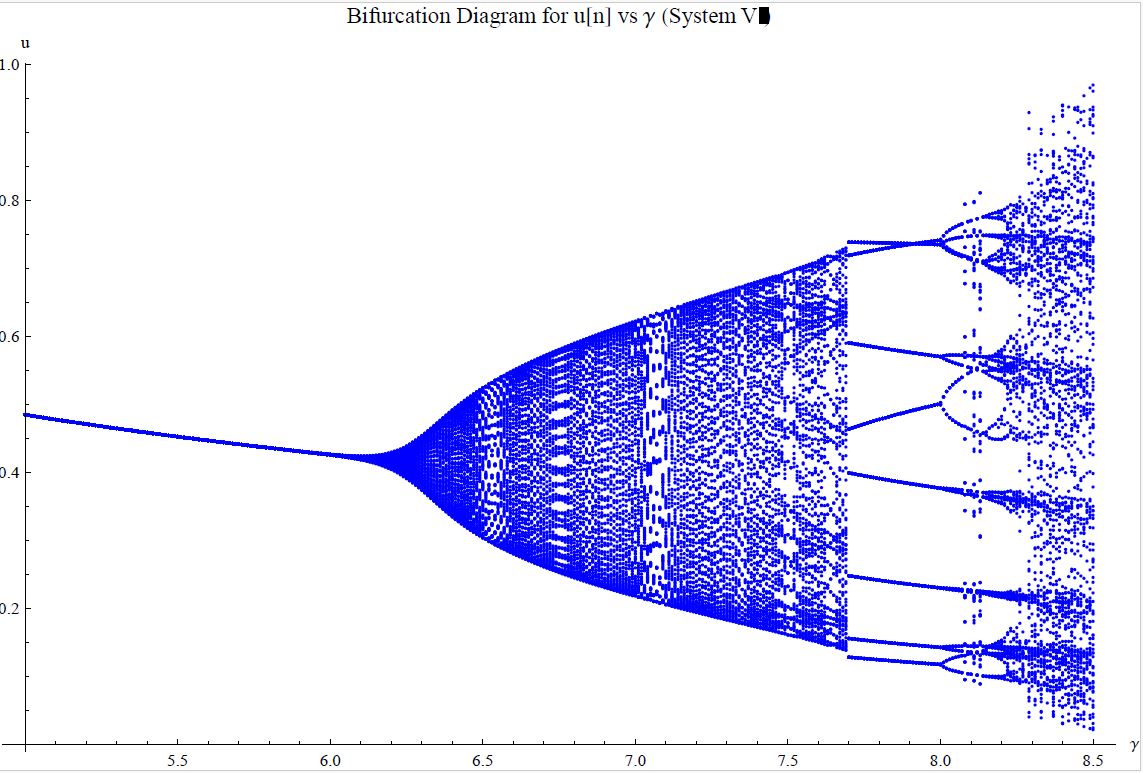}} \hspace{0.3in}
    \subfigure[]{\includegraphics[width=0.56\textwidth]{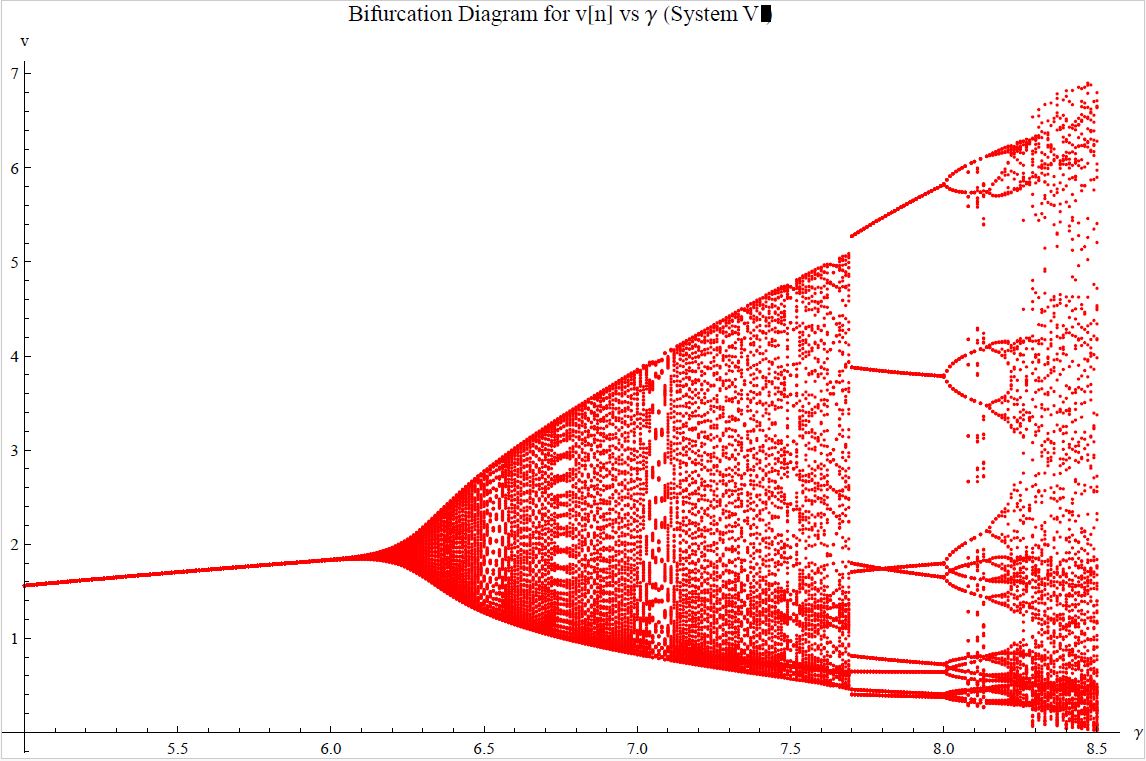}}\hspace{0.3in}
       \subfigure[]{\includegraphics[width=0.6\textwidth]{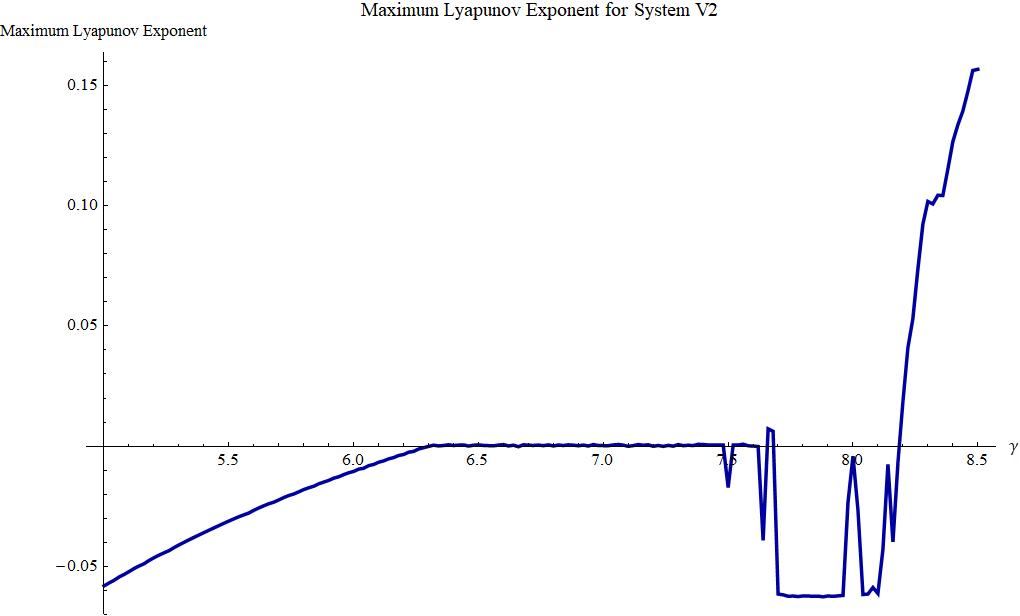}}
    \caption{\tiny Bifurcation diagrams for the system \ref{h2} with the parameters $r = 0.8, c =
2,$ and initial values $u^0 = 0.4, v^0 = 0.8$ when the bifurcation parameter $\gamma$ varying on the interval $5\leq\gamma\leq8.5$. In (c), the Maximum Lyapunov exponent corresponding to (a) and (b) are presented. The maximum Lyapunov exponent indicates that chaotic dynamics are observed when approximately \( 6.3 < \gamma < 7.478 \),  \ \ \( 7.522 < \gamma < 7.622 \), \ \ \( 7.658 < \gamma < 7.682 \) and \( 8.187 < \gamma < 8.5 \).}
    \label{diag2}
\end{figure}

\begin{figure}[h!]
    \centering
    \subfigure[\tiny$\gamma=6.2, \overline{u}\approx0.4170, q(\overline{u})\approx0.9929, (0.4, 0.8).$]{\includegraphics[width=0.45\textwidth]{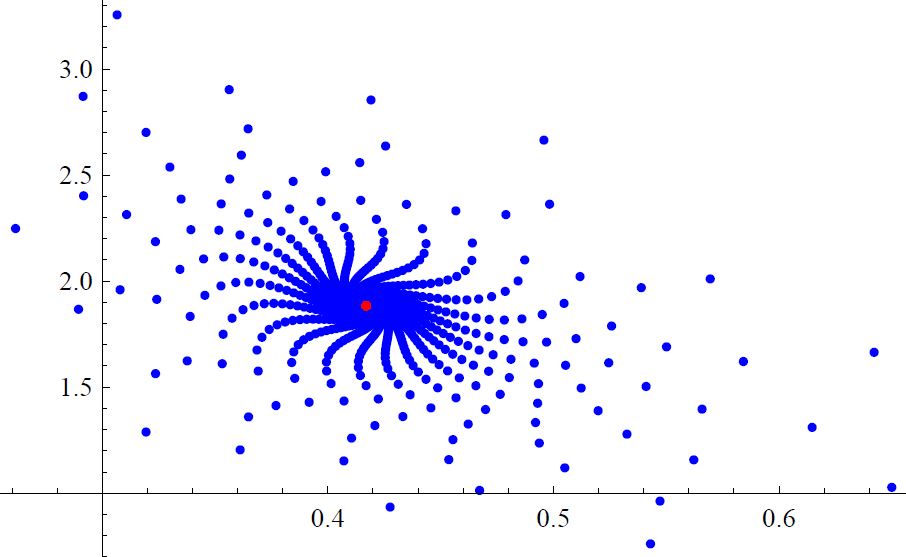}} \hspace{0.3in}
    \subfigure[\tiny$\gamma=6.4, \overline{u}\approx0.4082, q(\overline{u})\approx1.0059, (0.4, 0.8).$]{\includegraphics[width=0.45\textwidth]{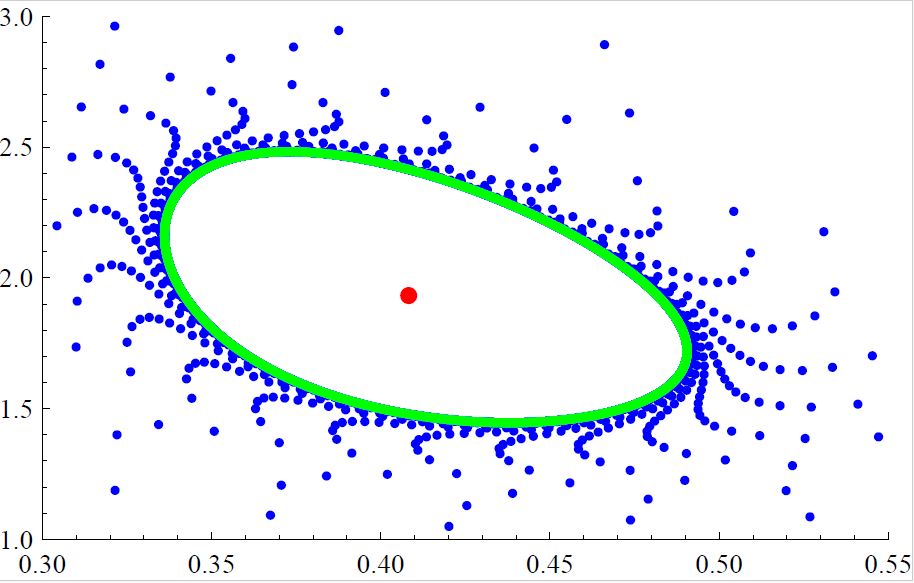}}
    \subfigure[\tiny$\gamma=6.4, \overline{u}\approx0.4082, q(\overline{u})\approx1.0059, (0.4, 2.1).$]{\includegraphics[width=0.45\textwidth]{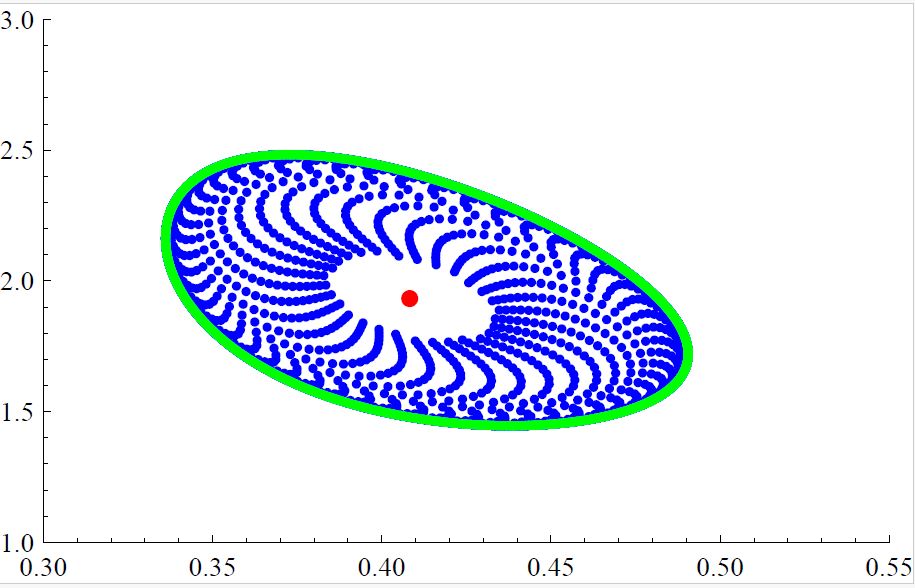}} \hspace{0.3in}
    \subfigure[\tiny$\gamma=7.6, (u^0,v^0)=(0.35, 3).$]{\includegraphics[width=0.45\textwidth]{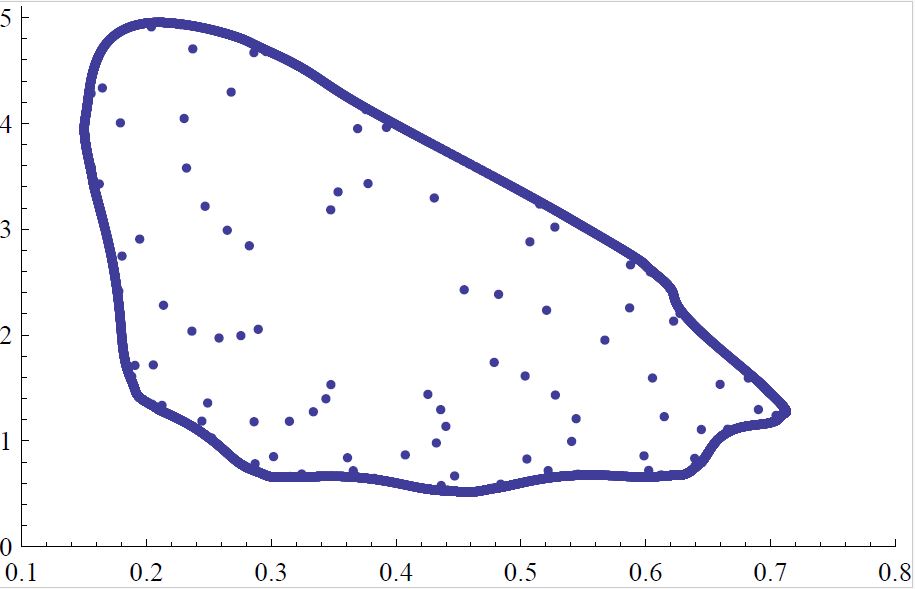}} \hspace{0.3in}\hspace{0.3in}
    \subfigure[\tiny$\gamma=8.2, (u^0,v^0)=(0.35, 3).$]{\includegraphics[width=0.43\textwidth]{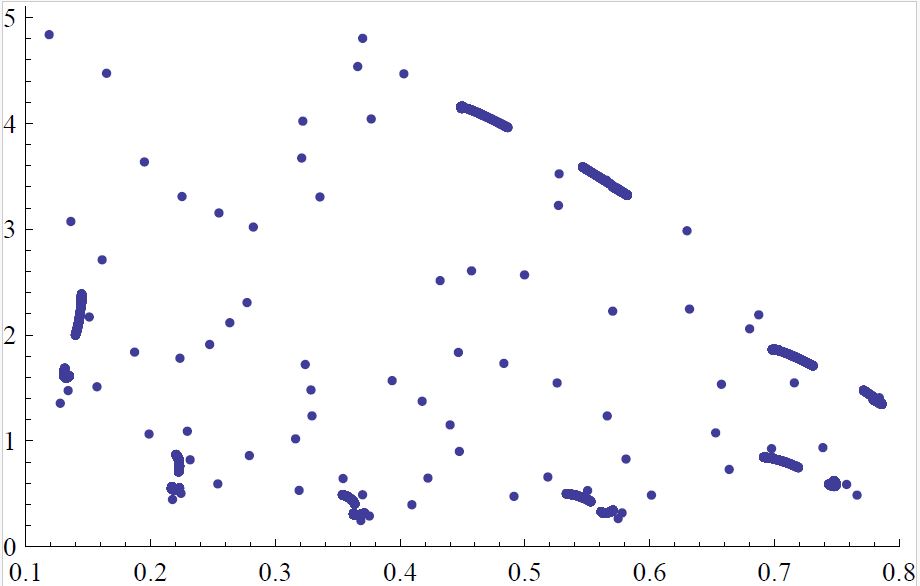}} \hspace{0.3in}\hspace{0.3in}
    \subfigure[\tiny$\gamma=8.7, (u^0,v^0)=(0.35, 3).$]{\includegraphics[width=0.43\textwidth]{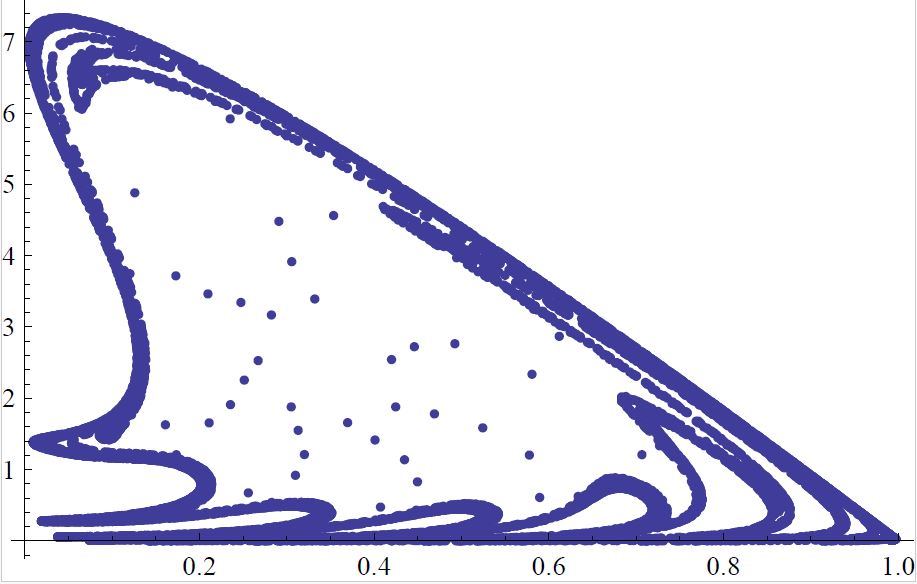}}
    \caption{\tiny Phase portraits for the system (\ref{h2}) with parameters $c=2$, $r=0.8$, and $n=10,000$. The red point represents the fixed point $\overline{E}$, while the green curve is an attracting invariant closed curve. In (a), $q(\overline{u}) < 1$, so the fixed point remains attracting. In (b), the initial point is chosen outside the invariant closed curve, whereas in (c), it is taken from inside. In (d), (e), and (f), we present trajectories for various values of \( \gamma \).}
    \label{fig2}
\end{figure}


\begin{thebibliography}{99}

\bibitem{Cah}
Cahit, K., Yasin, Y. {\em Stability, bifurcation analysis and chaos control in a discrete predator–prey system incorporating prey immigration}. Journal of Applied Mathematics and Computing, 70, 5213--5247 (2024).

\bibitem{Chatt}
Chattopadhyay, J., Sarkar, R. R., Mandal, S. {\em Toxin-producing plankton may act as a biological control for planktonic blooms—Field study and mathematical modelling}. J. Theor. Biol., 215(3), 333--344 (2002).

\bibitem{Chatt2}
 Chattopadhyay, J., Sarkar, R. R.,  Abdllaoui, A.  {\em A delay differential equation model on harmful algal blooms in the presence of toxic substances}. IMA Journal of Mathematics Applied in Medicine and Biology, 19(2), 137--161 (2002).

\bibitem{Chen3}
Chen, J., Chen, Y., Zhu, Z., Chen, F. {\em Stability and bifurcation of a discrete predator-prey system with Allee effect and other food resource for the predators}. Journal of Applied Mathematics and Computing, 69, 529--548 (2023).

\bibitem{Chen}
Chen, S., Yang, H., Wei, J. {\em Global dynamics of two phytoplankton-zooplankton models with toxic substances effect}. Journal of Applied Analysis and Computation, 9(3), 796--809 (2019).

\bibitem{Chen2}
Chen, L., Chen, F., Chen, L. {\em Qualitative analysis of a predator-prey model with Holling type II
functional response incorporating a constant prey refuge}. Nonlinear Analysis: Real World Applications, 11, 246--252 (2010).

\bibitem{Cheng}
Cheng, W., Wang, L. {\em Stability and Neimark-Sacker bifurcation of a semi-discrete population model}. Journal of Applied Analysis and Computation, 4(4), 419--435 (2014).

\bibitem{De}
Devaney, R. L. {\em An Introduction to Chaotic Dynamical Systems}. 2nd ed., Westview Press (2003).

\bibitem{Ed}
Edwards, A. M., Brindley, J. {\em Oscillatory behaviour in a three-component plankton population model}. Dynamics and Stability of Systems, 11(4), 347--370 (1996).

\bibitem{Ele}
Elettreby, M. F., Nabil, T., Khawagi, A. {\em Stability and Bifurcation Analysis of a Discrete Predator-Prey
Model with Mixed Holling Interaction}. Computer Modeling in Engineering and Sciences, 122(3), 907--921 (2020).

\bibitem{Hen}
Hening, A., Hieu, N., Nguyen, D., Nguyen, N. {\em Stochastic nutrient-plankton models}. Journal of Differential Equations, 376, 370--405 (2023).

\bibitem{Hong}
Hong, Y. {\em Global dynamics of a diffusive phytoplankton-zooplankton model with toxic substances effect and delay}. Math. Biosci. Eng., 19(7), 6712--6730 (2022).

\bibitem{Ko}
Ko, W., Ryu, K. {\em Qualitative analysis of a predator-prey model with Holling type II functional response incorporating a prey refuge}. J. Differential Equations, 231(2), 534--550 (2006).

\bibitem{Kuz}
Kuznetsov, Y. A. {\em Elements of Applied Bifurcation Theory}. 2nd ed., Springer-Verlag, New York (1998).

\bibitem{Li}
Li, B., Eskandari, Z., Avazzadeh, Z. {\em Strong resonance bifurcations for a discrete-time prey–predator model}. Journal of Applied Mathematics and Computing, 69, 2421--2438 (2023).

\bibitem{Mac}
Macdonald, J. C., Gulbudak, H. {\em Forward hysteresis and Hopf bifurcation in an NPZD model with application to harmful algal blooms}. Journal of Mathematical Biology, 87(3), 45 (2023).

\bibitem{Ma}
Ma, Z., Wang, S., Wang, T., Tang, H. {\em Stability analysis of prey-predator system with Holling type functional response and prey refuge}. Advances in Difference Equations, 2017:243 (2017).

\bibitem{Peng}
Peng, R., Shi, J. {\em Non-existence of non-constant positive steady states of two Holling type-II predator-prey systems: Strong interaction case}. J. Differential Equations, 247(3), 866--886 (2009).

\bibitem{Qiu}
Zhao, Q., Liu, S., Niu, X. {\em Dynamic behavior analysis of a diffusive plankton model with defensive and offensive effects}. Chaos, Solitons Fractals, 129, 94--102 (2019).

\bibitem{RSH}
Rozikov, U. A., Shoyimardonov, S. K. {\em Ocean ecosystem discrete time dynamics generated by $\ell$-Volterra operators}. International Journal of Biomathematics, 12(2), 1950015 (2019).

\bibitem{RSHV}
Rozikov, U. A., Shoyimardonov, S. K., Varro, R. {\em Plankton discrete-time dynamical systems}. Nonlinear Studies, 28(2), 585--600 (2021).

\bibitem{Sajan}
Sajan, S., Sasmal, S. K., Dubey, B. {\em A phytoplankton–zooplankton–fish model with chaos control: In the presence of fear effect and an additional food}. Chaos, 32(1), 013114 (2022).

\bibitem{SH}
Shoyimardonov, S. {\em Neimark-Sacker bifurcation and stability analysis in a discrete phytoplankton-zooplankton system with Holling type II functional response}. Journal of Applied Analysis and Computation, 13(4), 2048--2064 (2023).

\bibitem{Shang}
Chen, S., Chen, F., Chen, L. {\em Bifurcation analysis of an allelopathic phytoplankton model}. Journal of Biological Systems, 31(3), 1063--1097 (2023).

\bibitem{Tian}
Liao, T. {\em The impact of plankton body size on phytoplankton-zooplankton dynamics in the absence and presence of stochastic environmental fluctuation}. Chaos, Solitons Fractals, 157, 111617 (2022).

\bibitem{Wang}
Wang, J. {\em Spatiotemporal patterns of a homogeneous diffusive predator-prey system with Holling type III functional response}. J. Dyn. Diff. Equat., 29(4), 1383--1409 (2017).

\bibitem{Wing}
Wiggins, S. {\em Introduction to Applied Nonlinear Dynamical Systems and Chaos}. 2nd ed., Springer-Verlag, New York (2003).

\bibitem{Zhou}
Zhou, J., Mu, C. {\em Coexistence states of a Holling type-II predator-prey system}. J. Math. Anal. Appl., 369(2), 555--563 (2010).

\end{thebibliography}
\end{document}